\documentclass[12pt,a4paper]{amsart}
\usepackage{graphicx}
\usepackage{times}
\usepackage{latexsym}
\usepackage{amssymb}
\usepackage{hyperref}
\usepackage{xcolor}
\usepackage{tikz}
\usepackage[all]{xy}

\renewcommand{\baselinestretch}{1.13}
\newtheorem{prop}{Proposition}[section]
\newtheorem{cor}[prop]{Corollary}
\newtheorem{thm}[prop]{Theorem}
\newtheorem{lemma}[prop]{Lemma}
\newtheorem{dfn}[prop]{Definition}

\theoremstyle{definition}
\newtheorem{exm}[prop]{Example}

\newcommand{\C}[1]{{\mathcal #1}}
\newcommand{\B}[1]{{\mathbb #1}}
\newcommand{\G}[1]{{\mathfrak #1}}

\newcommand{\bs}[1]{\boldsymbol{#1}}
\renewcommand{\P}{\mathbb{P}}
\newcommand{\id}{\operatorname{\rm id}}
\newcommand{\Hom}{{\rm Hom}}
\newcommand{\ztwo}{{{\mathbb Z}\slash 2}}
\newcommand{\Pow}[1]{\boldsymbol{2}^{#1}}
\newcommand{\Fin}{{\tt Fin}}
\newcommand{\EPow}[1]{\Pow{#1}\setminus\{\emptyset\}}
\newcommand{\EFin}{{\tt Fin}\setminus\{\emptyset\}}
\newcommand{\Top}{{\tt Op}}
\newcommand{\Cl}{{\tt Cl}}
\newcommand{\Lat}{{\tt Lat}}

\newcommand{\op}{{\text{op}}}

\newcommand{\cov}{{\mathcal{C}ov}_{\rm fin}}
\newcommand{\covaux}{{\mathcal{A}ux}}
\newcommand{\auxp}{\widetilde{{\mathcal{A}ux}}}
\newcommand{\ocovf}{{\mathcal{O}\mathcal{C}ov}_{\rm fin}}

\newcommand{\pinf}{{\P^\infty\!(\ztwo)}}
\newcommand{\pN}{{\P^N\!(\ztwo)}}
\newcommand{\pM}{{\P^M\!(\ztwo)}}

\newcommand{\Sh}{{\rm Sh}(\pinf)}

\newcommand{\fSh}{{\rm Sh}_{\rm fin}(\pinf)}
\newcommand{\pfSh}{{\rm Sh}_{\rm fin}^{\rm ???}(\pinf)}
\newcommand{\pfShp}{{\widetilde{{\rm Sh}}}_{\rm fin}(\pinf)}

\newcommand{\ord}[1]{\underline{#1}}
\newcommand{\lmod}{\text{\bf Mod}}
\newcommand{\ncov}[1]{\boldsymbol{C}_{#1}}
\newcommand{\pPsi}{{\widetilde{\Psi}}}
\newcommand{\pPhi}{{\widetilde{\Phi}}}

\newcommand{\xra}[1]{\xrightarrow{\ \ {#1}\ \ }}

\renewcommand{\leq}{\leqslant}
\renewcommand{\geq}{\geqslant}

\newcommand{\mN}{{\mathbb{N}}}

\newcommand{\nc}[2]{\newcommand{#1}{#2}}
\nc{\bmlp}{\mbox{\boldmath$\left(\right.$}}
\nc{\bmrp}{\mbox{\boldmath$\left.\right)$}}
\nc{\LAblp}{\mbox{\LARGE\boldmath$($}}
\nc{\LAbrp}{\mbox{\LARGE\boldmath$)$}}
\nc{\Lblp}{\mbox{\Large\boldmath$($}}
\nc{\Lbrp}{\mbox{\Large\boldmath$)$}}
\nc{\lblp}{\mbox{\large\boldmath$($}}
\nc{\lbrp}{\mbox{\large\boldmath$)$}}
\nc{\blp}{\mbox{\boldmath$($}}
\nc{\brp}{\mbox{\boldmath$)$}}
\nc{\LAlp}{\mbox{\LARGE $($}}
\nc{\LArp}{\mbox{\LARGE $)$}}
\nc{\Llp}{\mbox{\Large $($}}
\nc{\Lrp}{\mbox{\Large $)$}}
\nc{\llp}{\mbox{\large $($}}
\nc{\lrp}{\mbox{\large $)$}}
\nc{\lbc}{\mbox{\Large\boldmath$,$}}
\nc{\lc}{\mbox{\Large$,$}}
\nc{\Lall}{\mbox{\Large$\forall\;$}}
\nc{\bc}{\mbox{\boldmath$,$}}

\nc{\LAbls}{\mbox{\LARGE\boldmath$[$}}
\nc{\LAbrs}{\mbox{\LARGE\boldmath$]$}}
\nc{\Lbls}{\mbox{\Large\boldmath$[$}}
\nc{\Lbrs}{\mbox{\Large\boldmath$]$}}
\nc{\lbls}{\mbox{\large\boldmath$[$}}
\nc{\lbrs}{\mbox{\large\boldmath$]$}}
\nc{\bls}{\mbox{\boldmath$[$}}
\nc{\brs}{\mbox{\boldmath$]$}}
\nc{\LAls}{\mbox{\LARGE $[$}}
\nc{\LArs}{\mbox{\LARGE $]$}}
\nc{\Lls}{\mbox{\Large $[$}}
\nc{\Lrs}{\mbox{\Large $]$}}
\nc{\lls}{\mbox{\large $[$}}
\nc{\lrs}{\mbox{\large $]$}}
\newcommand{\upset}[1]{{\uparrow\!\! #1}}
\newcommand{\downset}[1]{{\downarrow\!\! #1}}

\title[Finite closed coverings of compact quantum spaces]
{{\Large Finite closed coverings of}\\ \vspace{2.5mm}
{\Large compact quantum spaces}}

\author{Piotr~M.~Hajac}
\address{Instytut Matematyczny,
Polska Akademia Nauk,
ul.~\'Sniadeckich 8, Warszawa, 00-956 Poland\\
Katedra Metod Matematycznych Fizyki,
Uniwersytet Warszawski,
ul. Ho\.za 74, Warszawa, 00-682 Poland}
\email{http://www.impan.pl/\~{}pmh}

\author{Atabey~Kaygun} \address{Department of Mathematics and Computer
  Science, Bah\c{c}e\c{s}ehir University, \c{C}\i ra\u{g}an Cad.,
  Be\c{s}ikta\c{s} 34353 Istanbul, Turkey}
\email{atabey.kaygun@bahcesehir.edu.tr}

\author{Bartosz~Zieli\'nski}
\address{Instytut Matematyczny, Polska Akademia Nauk,
ul.~\'Sniadeckich 8, Warszawa, 00-956 Poland\\
Department of Theoretical Physics and Computer Science, University of \L{}\'od\'z,
Pomorska 149/153 90-236 \L{}\'od\'z, Poland}
\email{bzielinski@uni.lodz.pl}

\begin{document}
\baselineskip=16.1pt

\begin{abstract}
  We show that a projective space $\pinf$
  endowed with the Alexandrov topology is a classifying space for
  finite closed coverings of compact quantum spaces in the sense that
  any such a covering is functorially equivalent to a sheaf over this
  projective space. In technical terms, we prove that the category of
  finitely supported flabby sheaves of algebras is equivalent to the
  category of algebras with a finite set of ideals that intersect to
  zero and generate a distributive lattice.  In particular, the
  Gelfand transform allows us to view finite closed coverings of
  compact Hausdorff spaces as flabby sheaves of commutative
  C*-algebras over $\pinf$.
\end{abstract}

\vspace*{-15mm}\maketitle

\smallskip

\section*{Introduction}

\noindent{\bf Motivation.}
In the day-to-day practice of the mathematical art, one can see a
recurrent theme of reducing a complicated mathematical construct into
its simpler constituents, and then putting these constituents together
using  gluing datum that prescribes how these pieces consistently fit 
each other.  The (now) classical manifestation of such gluing
arguments in various flavours of geometry is the concept of a sheaf on a
topological space, or more generally on a topos.  Another
manifestation of such gluing arguments appeared in noncommutative
geometry as the description of a noncommutative space via a finite
closed covering. Here a covering is defined  as a distinguished
finite set of ideals that intersect to zero and generate a
distributive lattice~\cite{HKMZ:PiecewisePrincipalComoduleAlgebras}.

\noindent{\bf Main result.}
Following~\cite{HKMZ:PiecewisePrincipalComoduleAlgebras}, we
 express the gluing datum of a compact Hausdorff space
 as a sheaf of algebras
 over a certain universal
topological space, and extend it to the noncommutative setting.
This universal topological space is explicitly constructed as
the infinite $\mathbb{Z}/2$-projective space $\pinf$ endowed with the
Alexandrov topology. The
 advantages of our main theorem
over its predecessor
\cite[Cor.~4.3]{HKMZ:PiecewisePrincipalComoduleAlgebras}
are twofold.
First,
it considers coverings rather than topologically unnatural ordered
coverings. To this end, we need to construct more refined morphisms
between sheaves than natural transformations.
Next, as $\pinf $ is the colimit of all finite $\pN$'s, it takes
 care of all finite coverings at once.

\noindent{\bf Theorem~\ref{SheafandCoveringEquivalence}.}
{\em The category
of finite coverings of algebras is equivalent to the category
of finitely-supported flabby sheaves of algebras over~$\pinf$ whose
morphisms are
obtained by taking a certain quotient of the usual class of morphisms
enlarged by the actions of a specific family of endofunctors.}

\noindent{\bf Sheaves, patterns, and \boldmath$P$-algebras.}
The idea of using lattices to study closed
coverings of noncommutative spaces
has already been widely employed (see~\cite{l-g97}).
To afford a good C*-algebraic description, one considers
closed rather than open coverings. Therefore,
 a natural framework for coverings
uses sheaf-like objects defined on the lattice of closed subsets of a
topological space, or more generally, topoi modelled upon finite
closed coverings of topological spaces.  Interestingly, the original
definition of sheaves by Leray was given in terms of the lattice of
closed subspaces of a topological
space~\cite[p.~303]{Leray:SelectedWorksVolI}.
For
various reasons,
this definition changed in the subsequent years into the
nowadays standard open-set formulation.

Recently, however, a closed-set approach
reappeared in the form of sheaf-like objects called  {\em
  patterns}~\cite{Maszczyk:Patterns}.  We show in
Proposition~\ref{PatternsEquivalentToSheaves} that for our
combinatorial models based on finite
Alexandrov spaces, the distinction between sheaves and patterns is
immaterial.
Another reformulation of sheaves  over Alexandrov spaces
is given by the concept of a {\em $P$-diagram}. It is widely known
among commutative algebraists (e.g., see
\cite[Prop.~6.6]{BrunBrunsRomer:CohomologyOfPosets} and
\cite[p.~174]{Yuzvinski:CMRingsOfSections}) that any sheaf on an
Alexandrov space $P$ can be recovered from its
$P$-diagram (cf.~Theorem~\ref{PAlgebras} concerning $P$-algebras).  See
also~\cite{GerstenhaberSchack:CohomologyOfPresheaves} for a different
approach.

\noindent{\bf Outline.}
Section~1 is of preliminary nature. It is focused on
 explaining the emergence
of the projective space $\pinf$
as the classifying space of finite coverings. We show how
finite closed coverings of compact Hausdorff spaces naturally yield
 finite partition spaces with Alexandrov topology. Then we interpret
them as projective spaces $\mathbb{P}^N(\mathbb{Z}/2)$
and take the colimit
with \mbox{$N\rightarrow\infty$}.
We continue with analysing in detail the
topological properties of $\pinf$ to be ready for studying sheaves of
algebras over~$\pinf$. These are the key objects of Section~2
that is devoted to the main result of this paper.

\noindent{\bf Notation and conventions.}
Throughout the article we fix a ground field $k$ of an arbitrary
characteristic.  We assume that all algebras are over $k$ and are
associative and unital but not necessarily commutative.  We  use
$\B{N}$ and $\B{Z}$ to denote the set of natural numbers (zero included)
 and the set of
integers, respectively.  %We will assume $0\in \B{N}$.
The finite
set $\{0,\ldots,N\}$ is denoted by $\ord{N}$ for any natural
number~$N$.  However, the finite set $\{0,1\}$ when viewed as
the finite field of $2$ elements is denoted by $\ztwo$.  We
use $\Pow{X}$ to denote the set of all subsets of an arbitrary set~$X$.
If $\ord{x}$ is a sequence of elements from a set $X$, we write
$\kappa(\ord{x})$ to denote the underlying set of elements of
$\ord{x}$.  The symbol $|X|$  stands for
 the cardinality
of a set~$X$.

\section{Primer on lattices and Alexandrov topology}
\label{ProjectiveSpacesOverZ2}

We first recall definitions and simple facts about ordered sets and
lattices to fix notation.  Our main references on the subject are
\cite{Birkhoff:Lattices, BS:UniversalAlgebra,
  Stanley:EnumerativeCombinatoricsVol1}.

A set $P$ together with a binary relation $\leq$ is called {\em a
  partially ordered set}, or {\em a poset} in short, if the relation
$\leq$ is (i) reflexive, i.e., $p\leq p$ for any $p\in P$, (ii)
transitive, i.e., $p\leq q$ and $q\leq r$ implies $p\leq r$ for any
$p,q,r\in P$, and (iii) anti-symmetric, i.e., $p\leq q$ and $q\leq p$
implies $p=q$ for any $p,q\in P$.  If only the conditions (i)-(ii) are
satisfied we call $\leq$ a {\em preorder}.  For every preordered set
$(P,\leq)$ there is an opposite preordered set $(P,\leq)^\op$ given by
$P=P^\op$ and $p\leq^\op q$ if and only if $q\leq p$ for any $p,q\in
P$.

A poset $(P,\leq)$ is called {\em a semi-lattice} if for every $p,q\in
P$ there exists an element $p\vee q$ such that (i) $p\leq p\vee q$,
(ii) $q\leq p\vee q$, and (iii) if $r\in P$ is an element which
satisfies $p\leq r$ and $q\leq r$ then $p\vee q\leq r$.  The binary
operation $\vee$ is called {\em the join}.  A poset is called {\em a
  lattice} if both $(P,\leq)$ and $(P,\leq)^\op$ are semi-lattices.
The join operation in $P^\op$ is called {\em the meet}, and
traditionally denoted by $\wedge$.  One can equivalently define a
lattice $P$ as a set with two binary associative commutative and
idempotent operations $\vee$ and $\wedge$.  These operations satisfy
two absorption laws: $p = p\vee(p\wedge q)$ and $p= p\wedge(p\vee q)$
for any $p,q\in P$.  A lattice $(P,\vee,\wedge)$ is called {\em
  distributive} if one has $p\wedge(q\vee r) = (p\wedge q)\vee(p\wedge
r)$ for any $p,q,r\in P$.  Note that one can prove that the
distributivity of meet over join we have here is equivalent to the
distributivity of join over meet.

  Let $(P,\leq)$ be a preordered set, and let $\upset p = \{q\in P|\
  p\leq q\}$ for any $p\in P$. As a natural extension of notation, we
  define $\upset U:=\bigcup_{p\in U}\upset p$ for all $U\subseteq P$.
  The sets $U\subseteq P$ that satisfy $U=\upset U$
  are called {\em upper sets} or {\em dual order ideals}.  The
  topological space we obtain from a preordered set using the upper sets
  as open sets is called an {\em Alexandrov space}.  Note that a set
  $U$ is open in the Alexandrov topology if and only if for any $u\in
  U$ one has $\upset u\subseteq U$. Observe also that reversing the
  preorder exchanges the closed and open sets:  
\begin{lemma}\label{FlipTopology}
  Let $(P,\leq)$ be a preordered set.  A subset $C\subseteq P$ is
  closed in the Alexandrov topology of $P$ if and only if $C$ is open
  in the Alexandrov topology of the opposite preordered set
  $(P,\leq)^\op$.%\vspace{-1mm}
\end{lemma}
\begin{proof}
  Since $(P,\leq) = ((P,\leq)^\op)^\op$ and the statement is
  symmetric, we need to prove only one implication.  Assume $C$ is
  closed and let $x\in C$.  In order to prove that $C$ is open in the
  opposite Alexandrov topology, we need to show that $y\in C$ for any
  $y\leq x$.  Suppose the contrary that $y\leq x$ and $y\in C^c:=
  P\setminus C$.  Since $C^c$ is open in the Alexandrov topology of
  $(P,\leq)$ and $y\leq x$, we must have $x\in C^c$, which is a
  contradiction.
\end{proof}

\vspace{1.5mm}
\subsection{Projective spaces over $\ztwo$ as partition spaces}
~

\noindent
In~\cite{Sorkin:PosetSpaces}, Sorkin defined and investigated the order
structure on the spaces  we call here {\em partition spaces}. For
the lattice of subsets covering a space, the partition spaces play
a role analogous to the set of meet-irreducible elements of an
arbitrary finite
distributive lattice, i.e., they are much smaller than lattices
themselves while encoding important lattice properties.  Sorkin's
primary objective was to develop finite approximations for topological
spaces via their finite open coverings (see also
\cite{BBELLS:NoncommutativeLatticesAsFiniteApproximations,
  ELT:NoncommutativeLattices}).  Here we
will investigate spaces with finite closed rather than open coverings.  See also
\cite{Yuzvinski:CMRingsOfSections, Yuzvinski:FlasqueSheavesOnPosets}
for a more algebraic approach. We begin by analysing properties of
partition spaces.

\begin{dfn}\label{Sorkin}
  Let $X$ be a set and let $\C{C} = \{C_0,\ldots,C_N\}$ be a finite
  covering of $X$, i.e., let $\bigcup \C{C} := \bigcup_i C_i = X$.
  For any $x\in X$, we define its support $supp_\C{C}(x) =
  \{C\in\C{C}\ |\ x\in C\}$.  A~preorder $\preccurlyeq_\C{C}$ on $X$ is defined
  by $x\preccurlyeq_\C{C} y$ if and only if $supp_\C{C}(x) \supseteq
  supp_\C{C}(y)$.  We also define an equivalence relation $\sim_\C{C}$
  by letting $x\sim_\C{C} y$ if and only if
  $supp_\C{C}(x)=supp_\C{C}(y)$.  We call the quotient space
$X/\!\!\sim_\C{C}$ {\em the
    partition space} associated to the finite covering $\C{C}$ of
  $X$. This space is partially ordered by the relation induced
from $\preccurlyeq_{\C{C}}$.  
\end{dfn}

\begin{dfn}
  Let $X$ and $\C{C}$ be as before. We use $(X,\preccurlyeq_\C{C})$ to denote
  the set $X$ with its Alexandrov topology induced from the preorder
  relation $\preccurlyeq_\C{C}$ defined above.
\end{dfn}

\begin{exm}
  Consider a region on the 2-dimensional Euclidean plane covered by
  three disks in generic position, and the corresponding poset, as
  described below:
  
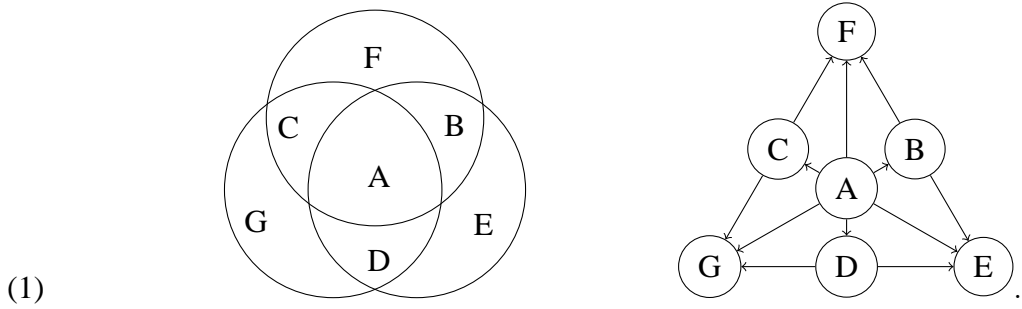
\begin{equation}
    \begin{tikzpicture}[scale=1.1]
      \draw (0:0) circle(1.3cm);
      \draw (60:1) circle(1.3cm);
      \draw (180:-1) circle(1.3cm);
      \coordinate [label=right:A] (A) at (30:.32);
      \coordinate [label=right:B] (B) at (33:1.42);
      \coordinate [label=left:C] (C) at (110:.8);
      \coordinate [label=right:D] (D) at (-72:.89);
      \coordinate [label=right:E] (E) at (-15:1.6);
      \coordinate [label=left:F] (F) at (66:1.75);
     \coordinate [label=left:G] (G) at (-150:.75);
    \end{tikzpicture}
   \hspace{2cm}
    \begin{tikzpicture}[scale=1.2]
      \node (center) at (0:0) [shape=circle,draw] {A};
      \node (top) at (90:1.732) [shape=circle,draw] {F};
      \node (middleright) at (30:.866) [shape=circle,draw] {B};
     \node (middleleft) at (150:.866) [shape=circle,draw] {C};
      \node (middlebottom) at (-90:.866) [shape=circle,draw] {D};
      \node (bottomleft) at (-150:1.732) [shape=circle,draw] {G};
      \node (bottomright) at (-30:1.732) [shape=circle,draw] {E};
     \draw[->] (center) to (middleleft);
      \draw[->] (center) to (middleright);
    \draw[->] (center) to (middlebottom);
      \draw[->] (center) to (top);
     \draw[->] (center) to (bottomleft);
      \draw[->] (center) to (bottomright);
      \draw[->] (middleright) to (top);
      \draw[->] (middleleft) to (top);
      \draw[->] (middleleft) to (bottomleft);
      \draw[->] (middleright) to (bottomright);
      \draw[->] (middlebottom) to (bottomleft);
      \draw[->] (middlebottom) to (bottomright);
    \end{tikzpicture}.
  \end{equation}
  Here an arrow $\to$ indicates the existence of an order relation
  between the source and the target.
\end{exm}

\begin{dfn}\label{PartitionTopology}
  Let $X$ be a set and $\C{C} = \{C_0,\ldots,C_N\}$ be a finite
  covering of $X$.  The covering $\C{C}$ viewed as a subbasis for
  closed sets induces a topology on $X$.  The space $X$ together with
  the topology induced from $\C{C}$ is denoted by $(X,\C{C})$.
\end{dfn}

\begin{prop}\label{SorkinHomeomorphism}
  Let $X$ be a set and let $\C{C}$ be a finite covering.  The
  Alexandrov topology defined by the preorder $\preccurlyeq_\C{C}$ coincides
  with the topology in Definition~\ref{PartitionTopology}.
\end{prop}

\begin{proof}
  We need to prove that a subset $L$ is closed in $(X,\C{C})$ if and
  only if it is closed in \mbox{$(X,\preccurlyeq_\C{C})$}.  By
  Lemma~\ref{FlipTopology} and the definition of Alexandrov
  topology, we see that $L$ is closed in \mbox{$(X,\preccurlyeq_\C{C})$}
  if and only if $L=\bigcup_{x\in L}\downset x$, where
   $\downset x := \{x'\in X|\ x'\preccurlyeq_\C{C} x\}$.  On the other hand, let
$C_x:=  \bigcap_{C\in supp_\C{C}(x)} C$.
 We have $x'\preccurlyeq_\C{C} x$
  if and only if $x'$ is covered by the same sets from $\C{C}$, or
  more.  In other words, $x'\preccurlyeq_\C{C}x$ if and only if $x'\in C_x$,
  so that $C_x=\downset x$. Finally, note that $L$ is closed in $(X,\C{C})$ if and
  only if $L=\bigcup_{x\in L} C_x$.  The result follows.
\end{proof}

\begin{cor}\label{corcanq}
  The canonical quotient map $\pi\colon (X,\C{C})\to (X/\!\!\sim_\C{C},\
  \preccurlyeq_\C{C})$ is a continuous map which is both open and closed.
\end{cor}
\begin{proof}
The above proposition allows us to replace $(X,\mathcal{C})$ by $(X,\preccurlyeq_\C{C})$
thus converting topological properties to preorder properties. Since $\pi$ is surjective and 
$x\preccurlyeq_\C{C} y$ if and only if 
$\pi(x)\preccurlyeq_\C{C}\pi(y)$, one easily verifies that  $\pi$ is continuous and open. To conclude
 that it is also closed, we apply Lemma~\ref{FlipTopology}.
\end{proof}

\begin{lemma}\label{lattisom}
  Let $\mathcal{C}$ be a finite covering of a set $X$.  Let
  $X/\!\!\sim_{\mathcal{C}}$ be the partition space associated with the
  covering $\mathcal{C}$ and $\pi\colon X\rightarrow
  X/\!\!\sim_{\mathcal{C}}$ be the canonical surjection on the
  quotient space. Denote by $\Lambda_{\mathcal{C}}$ the lattice of subsets of
  $X$ generated by the covering $\mathcal{C}$ and by
  $\Lambda_{\pi(\mathcal{C})}$ the lattice of subsets of
  $X/\!\!\sim_{\mathcal{C}}$ generated by $\pi(\mathcal{C}):=\{\pi(C)|\
  C\in\C{C}\}$. The following assignments
  \begin{align*}
    \hat{\pi}\colon\Lambda_{\mathcal{C}}\longrightarrow
    \Lambda_{\pi(\mathcal{C})}, \qquad &\lambda\longmapsto\pi(\lambda),\\
    \hat{\pi}^{-1}\colon \Lambda_{\pi(\mathcal{C})}\longrightarrow
    \Lambda_{\mathcal{C}}, \qquad &\lambda\longmapsto\pi^{-1}(\lambda).
  \end{align*}
define mutually inverse   lattice isomorphisms.
\end{lemma}
\begin{proof}
  Inverse images preserve set unions and intersections. Hence
  $\hat{\pi}^{-1}$ is a lattice morphism.  On the other hand,
though in general images
  preserve only unions, here we have
  \begin{equation}
 \pi(x)\in \pi(C_i)\quad\Leftrightarrow\quad x\in C_i
\end{equation}
 for any
  $i$.  It follows that
  \begin{align}
    \pi(x)\in\pi(C_{i_1})\cap\cdots\cap\pi(C_{i_k})
    &\Leftrightarrow x\in C_{i_1}\cap\cdots\cap C_{i_k}\nonumber\\
    {}&\Rightarrow \pi(x)\in\pi(C_{i_1}\cap\cdots\cap C_{i_k}).
  \end{align}
  In other words, $\pi(C_{i_1})\cap\cdots\cap\pi(C_{i_k})$ is a subset
  of $\pi(C_{i_1}\cap\cdots\cap C_{i_k})$.  As the containment in the
  other direction always holds, it follows that $\hat{\pi}$ is also a
  lattice morphism.  Finally, since $\pi$ is surjective and  $\pi^{-1}(\pi(C_i))=C_i$ for all $i$, one
  sees  that $\hat{\pi}^{-1}$ and 
  $\hat{\pi}$ are  inverse of each other.
\end{proof}

Let $\ord{N}$ be the set $\{0,\ldots,N\}$ for any $N\in\B{N}$.  The
projective space over a field $\Bbbk$ is denoted by $\B{P}^N(\Bbbk)$.
It is defined as the space $\Bbbk^{N+1}\setminus\{0\}$ divided by the diagonal
action of the non-zero scalars $\Bbbk^\times:=\Bbbk\setminus\{0\}$.
For $\Bbbk=\ztwo$, we obtain
\begin{equation}
  \pN:=\{(z_0,\ldots,z_N)\in (\ztwo)^{N+1}|\;\exists i\in\ord{N},\; z_i=1\}.
\end{equation}
The projective space $\pN$ has a natural poset structure: for any
$a = (a_i)_{i\in\ord{N}}$ and $b = (b_i)_{i\in\ord{N}}$ in $\pN$ we
write $a\leq b$ if and only if $a_i\leq b_i$ for any $i\in\ord{N}$.
We are now ready  to compare partition spaces with
${\Bbb Z}/2$-projective spaces with Alexandrov topology.
The following theorem is a direct generalization of
\cite[Prop.~4.1]{HKMZ:PiecewisePrincipalComoduleAlgebras}:
\begin{thm}\label{StartingPoint}
  Let  $\underline{\C{C}} = (C_0,\ldots,C_N)$ be a finite
  covering of $X$ with a fixed ordering on the elements of the covering.  Let
  $\chi_a$ be the characteristic function of a subset $a\subseteq
  \ord{N}$ viewed as an element of~$\pN$.  Then the map $\xi\colon X\to\pN$  defined by
  \begin{equation*}
 \xi(x) = \chi_{s(x)}\;, \quad
     s(x) = \{i\in\ord{N}|\ x\in C_i\}\;,
  \end{equation*}
  yields a morphism of preordered sets  $\xi\colon
  (X,\preccurlyeq_\C{C})^{\op}\to (\pN,\ \leq)$ and, consequently,  a
  continuous map between Alexandrov spaces.
Moreover, $\xi$ is both open
  and closed, and it factors
% through the quotient  $(X/\!\!\sim_\C{C},\preccurlyeq_\C{C})^{\op}$ 
as $\xi=\hat{\xi}\circ\pi$,
  where
  $\hat{\xi}:(X/\!\!\sim_\C{C},\preccurlyeq_\C{C})^{\op}\rightarrow(\pN,\
  \leq)$ is an embedding of Alexandrov topological spaces.
\end{thm}

%\vspace{1.5mm}

\subsection{Topological properties of partition
  spaces}\label{ProjectiveSpaces}~

\noindent
Let $\Pow{N}$ denote the set of all subsets of $\ord{N} =
\{0,\ldots,N\}$.  Both $\Pow{N}$ and $\EPow{N}$ are posets with respect
to the inclusion relation $\subseteq$.  For any non-empty subset
$a\subseteq \ord{N}$, one has a sequence $(a_0,\ldots,a_N)$ where
\begin{equation}
 a_i =
\begin{cases}
  1 & \text{ if } i\in a,\\
  0 & \text{ otherwise. }
\end{cases}
\end{equation}
In other words, the sequence $(a_0,\ldots,a_N)$ is the characteristic
function $\chi_a$ of the subset $a\subseteq \ord{N}$.  The assignment
$a\mapsto \chi_a$ determines a bijection between the set of non-empty
subsets of $\ord{N}$ and the projective space $\pN$. Its inverse
is defined as
\begin{equation}\label{invmap}
  \nu(z) := \{ i\in\ord{N}|\ z_i=1\},\quad
z=(z_i)_{i\in\ord{N}}\in (\ztwo)^{N+1}.
\end{equation}
With this bijection,
one has $(a_i)_{i\in\ord{N}}\leq (b_i)_{i\in\ord{N}}$ if and only if
$\nu((a_i)_{i\in\ord{N}})\subseteq \nu((b_i)_{i\in\ord{N}})$.  In
other words, we have the following:
\begin{prop}\label{usetscor}
  The map $\nu\colon \pN\to \EPow{N}$ is an isomorphism of posets,
  and thus a homeomorphism of Alexandrov spaces.
\end{prop}

\begin{dfn}
  For any $i\in\ord{N}$ and any non-empty subset $a\subseteq\ord{N}$, we
  define open sets
  \begin{equation*}
    \B{A}_i^N=\{(z_0,\ldots,z_N)\in(\ztwo)^{N+1}\;|\;z_i=1\} = \upset \chi_{\{i\}}
    \quad \text{ and }\quad
    \B{A}^N_a := \bigcap_{i\in a}\B{A}_i^N = \upset \chi_a\;.
  \end{equation*}
\end{dfn}
\noindent
 Note that the sets $\B{A}_i^N$ form a subbasis for the Alexandrov topology of $\pN$.
 For brevity, when there is no risk of confusion we omit the
  superscripts and write $\B{A}_i$ and $\B{A}_a$ instead of
  $\B{A}_i^N$ and $\B{A}_a^N$.

\begin{lemma}\label{Canonical}
  For all $N\geq 0$, the map $\phi_N\colon \pN\longrightarrow
  \P^{N+1}(\ztwo)$ defined by
  \begin{equation}\label{embedding}
    \phi_N(z_0,\ldots,z_N):= (z_0,\ldots,z_N,0)
  \end{equation}
  is an embedding of topological spaces.
\end{lemma}
\begin{proof}
  The fact that the maps $\phi_{N}$ are injective is obvious.  They
  are also continuous since we have
  \begin{equation}
    \phi_{N}^{-1}(\B{A}_i^{N+1})
    =\begin{cases}
      \B{A}_i^N  & \text{ if } i\leq N,\\
      \emptyset & \text{ if } i=N+1.
    \end{cases}
 \end{equation}
Finally, $\phi_{N}$'s yield homeomorphisms between their domains and their images because 
  \begin{equation}
    \phi_N(\pN)\cap \B{A}_i^{N+1}
    =\begin{cases}
      \phi_N(\B{A}_i^N) & \text{ if } i\in\ord{N},\\
      \emptyset & \text{ otherwise, }
    \end{cases}
  \end{equation}
  for the open subsets in the subbasis of the Alexandrov topology.
\end{proof}

The maps $\phi_N\colon \pN\to \P^{N+1}(\ztwo)$ form a direct
system of maps of Alexandrov spaces.  Hence we can define the infinite projective space 
$\P^{\infty}(\ztwo)$ as a direct limit:
\begin{dfn}
  $\pinf := \underset{\underset{{N\geq 0} }{\longrightarrow}}{\rm lim}\;\pN$.
\end{dfn}
\noindent
We can represent the points of $\pinf$ as infinite sequences
$\{(z_i)_{i\in\B{N}}\;|\;z_i\in\ztwo\}$ where the number of non-zero
terms is finite and greater than zero.  We can also view $\pinf$ as the colimit of all finite $\pN$'s. 
The canonical morphisms of the
colimit $i_N\colon \pN\to\pinf$ send a finite sequence
$(z_0,\ldots,z_N)$ to the infinite sequence
$(z_0,\ldots,z_N,0,0,\ldots)$ obtained from the finite sequence by
padding it with countably many $0$'s.  The topology on the colimit is
the topology induced by the maps $\{i_N\}_{N\in\B{N}}$.

We also have a natural poset structure on $\pinf$ where
$(a_i)_{i\in\B{N}}\leq (b_i)_{i\in\B{N}}$ if and only if $a_i\leq b_i$
for any $i\in\B{N}$. This poset structure coincides with the poset structure
of the set of all finite subsets of $\B{N}$. Denote  the set of all finite subsets of $\B{N}$ by $\Fin$.
 One can extend the bijection
$\nu\colon \pN\to \EPow{N}$ (see~\eqref{invmap}) 
to a bijection  $\nu\colon
\pinf\to \EFin$.  The inverse of $\nu$ is given by the assignment
$a\mapsto \chi_a := (a_i)_{i\in \B{N}}$ that is defined as
\begin{equation} \label{homeo}
a_i =
  \begin{cases}
    1 & \text{ if } i\in a,\\
    0 & \text{ otherwise, }
  \end{cases}
\end{equation}
for any $a\in\Fin$.  The map $\nu\colon \pinf\to \EFin$ is an
isomorphism of posets, and therefore the Alexandrov spaces $\pinf$ and
$\EFin$ are homeomorphic.

Thus we have two possibly different
topologies on $\pinf$: one coming from the preorder structure, and the
other coming from the colimit. However, we check that they coincide.

\begin{thm}\label{TopologicalProperties}
  The following statements hold:
\vspace*{-2mm}
  \begin{enumerate}
  \item\label{ColimitTopology} The Alexandrov topology and the colimit
    topology on $\pinf$ are the same.
  \item\label{T1NotT0} The spaces $\pN$ are $T_0$ but not $T_1$ for
    any $N=1,\ldots,\infty$.
  \item\label{Connected} $\pN$ is a connected topological space for
    any $N=0,1,\ldots,\infty$.
  \item \label{CompactlyGenerated} The topology on $\pinf$ is
    compactly generated.
  \end{enumerate}
\end{thm}

\begin{proof}
  For any $i\in\B{N}$ and $a\in\EFin$, we define
  \begin{equation}
    \B{A}^\infty_i := \upset \chi_{\{i\}}
    \quad \text{ and } \quad
    \B{A}_a^\infty := \bigcap_{i\in a} \B{A}_i^\infty = \upset \chi_a
  \end{equation}
  which are open in the Alexandrov topology.

  Proof of \eqref{ColimitTopology}: Let $i_N\colon \pN\to \pinf$ be
  the structure maps of the colimit.  We need to prove that an open
  set in one topology is open in the other, and vice versa.  The set
  $\{\B{A}^\infty_a|\ a\in\EFin\}$ is a basis for the Alexandrov
  topology since each $\B{A}_a^\infty$ is an upper set.
  Then
  \begin{equation}
 i_N^{-1}(\B{A}^\infty_a) =
  \begin{cases}
    \B{A}^N_{a} & \text{ if } a\subseteq \ord{N},\\
    \emptyset & \text{ if } a\nsubseteq \ord{N}
  \end{cases}
  \end{equation}
  is an open set in $\pN$ for any $N\geq 0$ and $a\in\EPow{N}$.
  Therefore, every open set in the Alexandrov topology is open in the
  colimit topology.  Now, assume $U\subseteq \pinf$ is open in the
  colimit topology.  We can assume every sequence in $\pinf$ is of the
  form $\chi_a$ for a unique $a\in\EFin$ since $z = \chi_{\nu(z)}$ for
  all $z\in\pinf$.  Next, assume $\chi_a\in U$ and we have $\chi_a\leq
  \chi_b$ for some $\chi_b\in\pinf$.  We need to show that $\chi_b\in
  U$.  Since $a\subseteq b\in\EFin\subset\text{\boldmath $2$}^\B{N}$, we have a natural number $N = \max(b)\geq
  \max(a)\geq1$.  Moreover, we have an inequality
  \begin{equation}
 i_N^{-1}(\chi_a)= \chi_a \leq \chi_b =
    i_N^{-1}(\chi_b)
\end{equation}
 in $\pN$.  As
  $i_N^{-1}(U)$ is open in the Alexandrov topology of $\pN$, 
it follows that $\chi_b\in i_N^{-1}(U)$ in $\pN$, which in turn implies
  $\chi_b\in U$.

  Proof of \eqref{T1NotT0}: Let $p,q\in\pN$, $p\neq q$. Then
  $\nu(p)\neq\nu(q)$.  Let us suppose without loss of generality that
  $i\in\nu(p)$ and $i\notin\nu(q)$. Then $q\notin\upset p$ which
  proves that $\pN$ is $T_0$.  On the other hand if $p\leq q$ then for any
  open set $U\subseteq\pN$ such that $p\in U$ also $q\in U$ (as $U$ is
  an upper set). It follows that $\pN$ is not $T_1$.

  Proof of \eqref{Connected}: Suppose there exists a non-empty subset
  $V \subsetneq \pN$ that is both open and closed.
Let $\chi_a\in V$ and
  $\chi_b\in\pN\setminus V$. Then, because
  $V$ and
  $\pN\setminus V$ are open, we have $\chi_{a\cup b}\in V$ and
  $\chi_{a\cup b}\in \pN\setminus V$,
   which
  is a contradiction.

  Proof of \eqref{CompactlyGenerated}: In order to prove our
  assertion, we need to show that for any $a\in \EFin$ the set
  $\B{A}^\infty_a$ is compact. Let $a\in\EFin$ and suppose that $\C{U}
  = \{U_i\}_{i\in I}$ is an open covering of $\B{A}^\infty_a$.  Since
   $\chi_a\in \B{A}_a^\infty$ and  $\C{U}$
  is a covering, there exists $j\in I$ such that $\chi_a\in U_j$.
  As $U_j$ is open in the Alexandrov topology, we obtain $\upset
  \chi_a = \B{A}_a^\infty \subseteq U_j$.  Consequently, for any
  finite subset $\alpha$ of $\EFin$, the set $\bigcup_{a\in\alpha}
  \B{A}^\infty_a$ is compact.  The result follows.
\end{proof}

\vspace{1.5mm}

\subsection{Continuous maps between partition spaces}
\label{ContinuousMapsOfTheProjectiveSpaces}~

\noindent
In what follows in this subsection, 
unless explicitly stated otherwise, $N$ will be a
 natural number or~$\infty$.  Accordingly, the set
$\{0,\ldots,N\}$ will be a finite set or will be $\B{N}$ if
$N=\infty$.  For example, a permutation
$\sigma:\{0,\ldots,N\}\rightarrow \{0,\ldots,N\}$ is either a finite
permutation or an arbitrary bijection $\B{N}\rightarrow\B{N}$.

Let $\Top(\pN)$ be the lattice of open subsets of~$\pN$.  It is
obvious that any continuous map $f:\pN\rightarrow\pM$ defines a
morphism between lattices of open sets of the form
$\G{X}_f:\Top(\pM)\longrightarrow\Top(\pN)$, where
\begin{equation}
  \G{X}_f(U):= f^{-1}(U).
\end{equation}
Conversely, we have the following:
\begin{prop}\label{PosetToFunction}
  Let $M$ and $N$ be finite natural numbers or $\infty$. Let
  $\G{X}:\Top(\pM)\longrightarrow\Top(\pN)$ be a lattice
  morphism with the property that
\begin{subequations}
\begin{gather}\label{boundcond}
  \bigcup_{i\in\{0,\ldots,M\}} \G{X}(\B{A}_i^M)=\pN,\\
  \bigcap_{i\in a}\G{X}(\B{A}_i^M)=\emptyset,\quad\text{for all
    infinite}\quad a\subseteq\{0,\ldots,M\}.
\label{emptycond}
\end{gather}
\end{subequations}
Then there exists a unique continuous function
$f_{\G{X}}:\pN\rightarrow\pM$ such that, for
all open subsets $U\subseteq\pM$, we have
$\G{X}(U)=f_{\G{X}}^{-1}(U)$.
\end{prop}
\begin{proof}
  We define a map
  $f_{\G{X}}:\pN\rightarrow\pM$ as
\begin{equation}
  f_{\G{X}}:z\longmapsto \chi_a
  ,\quad\text{where}\quad a:=\{i\in\{0,\ldots,M\}\;|\;z\in \G{X}(\B{A}^M_i)\}.
\end{equation}
We observe that $a$ is non-empty due to the
condition~\eqref{boundcond}, and finite due to the
condition~\eqref{emptycond}.  By definition, 
\begin{equation}\label{equivalences}
z\in
f_{\G{X}}^{-1}(\B{A}_i^M) \Leftrightarrow
f_{\G{X}}(z)\in \B{A}_i^M \Leftrightarrow i\in
\nu(f_{\G{X}}(z))\Leftrightarrow z\in \G{X}(\B{A}_i^M).
\end{equation}
This proves the continuity of $f_{\G{X}}$ because the sets $\B{A}_i^M$ form a subbasis
of the Alexandrov topology. The uniqueness follows from combining
\eqref{equivalences} with the fact that knowing for all $i$'s
whether or not $z'\in\B{A}_i^M$ determines $z'\in\pM$.
\end{proof}
\noindent
Note that the conditions~\eqref{boundcond} and \eqref{emptycond} are
satisfied for $\G{X}_f$ for any continuous $f$ because $\bigcap_{i\in
  a}\B{A}_i=\emptyset$ for any infinite $a$, and $f^{-1}$ preserves
infinite unions and intersections.

Finally, in order to characterize in Theorem~\ref{homeothm} the homeomorphisms between the
projective spaces $\pN$, we will need the following technical lemma.
\begin{lemma}\label{condhom}
  Let $N$ and $M$ be finite natural numbers or $\infty$.  Let $f\colon
  \pN\to \pM$ be a continuous map of Alexandrov spaces.
  \begin{enumerate}
  \item If $f$ is injective, then $|\nu(z)|\leq |\nu(f(z))|$ for any
    $z\in \pN$.
  \item If $f$ is a homeomorphism, then $|\nu(z)|= |\nu(f(z))|$ for any
    $z\in \pN$.
  \end{enumerate}
\end{lemma}

\begin{proof}
  Observe that for any $z\in \pN$ one can compute $|\nu(z)|$
  as
 \begin{equation}\label{19}
 |\nu(z)| = \max\{n\in\ord{N}|\ a_1<\cdots<a_n=z,\ a_i\in \pN\}.
\end{equation}
Here the symbol $x<y$ means $x\leq y$ and $x\neq y$. On the other hand, any map between spaces equipped with preorders is continuous with respect to
the Alexandrov topologies induced by these preorders  if and only if
it is monotonous, i.e., it preserves the preorders. 
 Therefore, if $f$ is continuous (i.e.~preorder preserving) and injective, then \eqref{19} implies that
  $|\nu(z)|\leq |\nu(f(z))|$ for any $z\in \pN$. 
Finally, if $f$ is homeomorphism, then we also have $|\nu(f(z))|\leq |\nu(f^{-1}(f(z)))|=|\nu(z)|$,
so that $|\nu(z)|= |\nu(f(z))|$ for any   $z\in \pN$.
\end{proof}
\noindent
Note that any continuous bijection between any two finite homeomorphic topological spaces (not necessarily Hausdorff) 
is always a homeomorphism. Hence, for any finite $N$, a continuous bijection from $\pN$ to $\pN$ is automatically
a homeomorphism, so that it enjoys the property (2) of the lemma above.

\begin{thm}\label{homeothm}
Let $N$ be a finite natural number or $\infty$.
  A map $f\colon\pN\rightarrow \pN$
  is a homeomorphism if and only if there exists a bijection
  $\sigma\colon \ord{N}\to \ord{N}$ such that
$f(\chi_a) = \chi_{\sigma(a)}$ for any subset~$a\subseteq\B{N}$.
\end{thm}

\begin{proof}
 We consider a bijection
  $\sigma\colon\ord{N} \to \ord{N}$.  It induces a bijection of the
  form 
\begin{equation}
f_\sigma\colon\pN\ni\chi_a\longmapsto \chi_{\sigma(a)}\in\pN
\end{equation}
 with the inverse
  $(f_\sigma)^{-1}=f_{\sigma^{-1}}$.  Since
  $f_\sigma(\B{A}_i)=\B{A}_{\sigma(i)}$ for all $i$ and the set
of all $\B{A}_i$'s is a subbasis for the topology of $\pN$, we conclude that $f_\sigma$
  is a homeomorphism. 

 Conversely, assume we have a homeomorphism
  $f\colon \pN\to \pN$.  Consider $\ell\subseteq \ord{N}$ and
  $\chi_\ell\in \pN$.  By Lemma~\ref{condhom}, the function $f$
  satisfies $|\nu(z)| = |\nu(f(z))|$ for any $z\in \pN$.  Therefore, applying
the support map $\nu$ to both sides of the equality $f(\chi_{\{i\}})=\chi_{\{\sigma(i)\}}$
 determines a unique map $\sigma\colon\ord{N}\to\ord{N}$ satisfying this equality:
\begin{equation}
\nu(f(\chi_{\{i\}}))=\nu(\chi_{\{\sigma(i)\}})=\{\sigma(i)\}.
\end{equation}
The inverse of thus defined map is given by the formula $\{\sigma^{-1}(i)\}=\nu(f^{-1}(\chi_{\{i\}}))$. 

Next, we proceed by induction on the cardinality of $a\subseteq\ord{N}$.   Assume that we have
  already proven that $f(\chi_a)=\chi_{\sigma(a)}$ for all
  $a$ such that $0<|a|\leq n$.  Pick $a\subseteq \ord{N}$
  with $|a|=n$ and $j\in\ord{N}\setminus a$.  Then, since the continuity of $f$ is
equivalent to $f$ being monotonous, we obtain $\chi_{\sigma(a)}=f(\chi_a)\leq f(\chi_{a\cup\{j\}})$ in
  $\pN$. Hence  $ f(\chi_{a\cup\{j\}})=\chi_{\sigma(a)\cup \ell}$ for some $\ell \subseteq \ord{N}$. 
On the other hand, by Lemma~\ref{condhom}, we see that
  $|\nu(f(\chi_{a\cup\{j\}}))|=n+1$, so that $\ell=\{k\}$
   for some
  $k\notin\sigma(a)$.  It remains to prove that
  $k=\sigma(j)$.  By definition
  $\chi_{\sigma(a)\cup\{k\}}\in \B{A}_{k}$.  Therefore,
  $\chi_{a\cup\{j\}}\in f^{-1}(\B{A}_k)=\B{A}_{\sigma^{-1}(k)}$, whence
   $\sigma^{-1}(k)\in a\cup\{j\}$. Combining this with  $\sigma^{-1}(k)\notin
  a$ yields $\sigma(j)=k$, as needed.
\end{proof}

We end this subsection by introducing a monoid that acts on $\pinf$
by continuous maps and is pivotal in our classification theorem. It
is a monoid that labels all finite sequences that can be  formed
from a given finite set.

\begin{dfn}\label{FiniteFibers}
  A surjection $\alpha\colon \B{N}\to \B{N}$ is called {\em tame} if
  \begin{enumerate}
  \item $\alpha^{-1}(i)$ is finite for any $i\in\B{N}$,
  \item $|\alpha^{-1}(i)|>1$ for finitely many $i\in\B{N}$.
\end{enumerate}
We denote the
  monoid of all such tame surjections by $\C{M}$.
\end{dfn}
\noindent
It is clear that the composition of any two tame surjections is again
a tame surjection, and that the monoid is generated by bijections and
the following tame surjection:
\begin{equation}
\partial(i) = \begin{cases}
  i & \text{ if } i=0,\\
  i-1 & \text{ if } i>0\ .
\end{cases}
\end{equation}

We can view the elements of $\pinf$ as maps from $\mN$ to $\ztwo$,
and on such maps the monoid $\C{M}$ acts by pullbacks.  Moreover,
the tameness
property ensures that such pullbacks preserve $\pinf$ and 
\begin{equation}\label{fSigma}
  f_\alpha(\chi_a):= \alpha^*(\chi_a)=\chi_{\alpha^{-1}(a)}\quad\text{ for all }
  a\in\Fin\setminus\emptyset,
\end{equation}
guarantees that they are morphisms of posets. Thus
we obtain an action of  $\C{M}$ on $\pinf$ by maps continuous in the
Alexandrov topology.  Observe that this pullback representation of the
monoid $\C{M}$ is faithful.
Note also that Theorem~\ref{homeothm} can be rephrased to link the bijections from $\ord{N}$
to  $\ord{N}$ with the homeomorphisms from $\pN$ to $\pN$ by the formula
$f(\chi_a) = \chi_{\sigma^{-1}(a)}$ for any subset~$a\subseteq\B{N}$. This makes Theorem~\ref{homeothm}
 compatible with
\eqref{fSigma}.
\vspace{1.5mm}

\subsection{The lattice of open subsets of $\pinf$}\label{latopen}
~

\noindent
In this subsection, we provide  a direct generalization of
\cite[Subsect.~2.2]{HKMZ:PiecewisePrincipalComoduleAlgebras}
 needed
to upgrade the flabby-sheaf classification of ordered $N$-coverings
\cite[Cor.~4.3]{HKMZ:PiecewisePrincipalComoduleAlgebras} to a
classification of arbitrary finite ordered coverings we arrive at in
Lemma~\ref{Equivalence}.
\begin{lemma}\label{LatticeMorphism}
  Let $(\Lat(A), \cap, +)$ denote the lattice of all ideals in
  an algebra $A$.
  Assume that 
$(I_i)_{i\in\mN}$ is a sequence of
  ideals such that only finitely many of them are different from~$A$
and that the lattice they generate is distributive.
  Then, for any open subset $U\subseteq \pinf$, the map
  given by
  \begin{equation}\label{Rdef}
    R^{(I_i)_i}(U) := \bigcap_{a\in \nu(U)}\sum_{i\in a} I_i
  \end{equation}
  defines a morphism of lattices
  $R^{(I_i)_i}\colon\Top(\pinf)\to \Lat(A)$.
\end{lemma}

\begin{proof}
  Since the map $\nu\colon \pinf\to \Fin\setminus\{\emptyset\}$
given by (\ref{invmap}) for $N=\infty$
 is a bijection, we have
  \begin{equation}
 \nu(U_1\cap U_2) = \nu(U_1)\cap \nu(U_2) \quad\text{ and }\quad
     \nu(U_1\cup U_2) = \nu(U_1)\cup\nu(U_2)
  \end{equation}
  for any $U_1,U_2\in\Top(\pinf)$.
  In order to prove that $R^{(I_i)_i}$ is a morphism
  of lattices, we need to show that
  \begin{equation}\label{lattmor}
    R^{(I_i)_i}(U_1\cap U_2)=R^{(I_i)_i}(U_1) +
R^{(I_i)_i}(U_2),\quad
    R^{(I_i)_i}(U_1\cup U_2)=R^{(I_i)_i}(U_1)\cap R^{(I_i)_i}(U_2),
  \end{equation}
  for all $U_1,U_2\in\Top(\pinf)$.  
Note that the latter equality above is trivially satisfied.

  To prove the former identity, first we observe that for all upper sets
  $\alpha_1,\alpha_2\subseteq\Fin$ 
  \begin{equation}\label{cup}
    \alpha_1\cap\alpha_2=
\{a_1\cup a_2\;|\;a_1\in\alpha_1,a_2\in\alpha_2\}.
  \end{equation}
  Indeed, 
since $a_1\subseteq a_1\cup a_2$ and $a_2\subseteq a_1\cup a_2$, we
  see that the left hand side contains the right hand side.  The other
  inclusion follows from the fact that
  %$\alpha_1\cap\alpha_2\subseteq\alpha_1$ and
  %$\alpha_1\cap\alpha_2\subseteq\alpha_2$ and 
$a=a\cup a$.  Next, we note that although the
  intersection in~\eqref{Rdef} is potentially infinite, there
  are only finitely many ideals different from~$A$. This fact allows us
to use the
  distributivity of the lattice generated by the ideals~$I_i$.
Furthermore, since $\nu$ is a homeomorphism (see below~\eqref{homeo})
with respect to the Alexandrov topologies (open sets are upper sets), 
we can use \eqref{cup} to conclude that 
\[
\forall\;U_1,U_2\in\Top(\pinf):\;\;
\nu(U_1\cap U_2)=\nu(U_1)\cap\nu(U_2)
=\{a_1\cup a_2\;|\;a_1\in\nu(U_1),a_2\in\nu(U_2)\}.
\]
Combining all this together,
we obtain:
  \begin{align}
    \bigcap_{a\in \nu(U_1\cap U_2)}\sum_{i\in a} I_i
    = & \bigcap_{a\in\nu(U_1)\cap\nu(U_2)}
 \sum_{i\in a}I_i\nonumber\\
    = & \bigcap_{a_1\in\nu(U_1)}\bigcap_{a_2\in\nu(U_2)}
\left(\sum_{i\in a_1} I_i+\sum_{j\in a_2} I_j\right)\nonumber\\
= & \bigcap_{a_1\in\nu(U_1)}\left(
\sum_{i\in a_1} I_i+\bigcap_{a_2\in\nu(U_2)}\sum_{j\in a_2} I_j\right)\nonumber\\
    = & \bigcap_{a_1\in\nu(U_1)}\sum_{i\in a_1}I_i
    +\bigcap_{a_2\in\nu(U_2)}\sum_{j\in a_2}I_j.
  \end{align}
The result follows.
\end{proof}
\vspace*{1.5mm}

\subsection{Sheaves and patterns on Alexandrov spaces}\label{SheavesAndPatterns}
~

\noindent
In~\cite{Maszczyk:Patterns}, Maszczyk defined the topological dual of a
sheaf, called {\em a pattern}, akin to Leray's original definition of
sheaves~\cite[p.~303]{Leray:SelectedWorksVolI} using closed sets
instead of open sets.  A pattern is a sheaf-like object defined on the
category of closed subsets $\Cl(X)$ of a topological space $X$ with
inclusions.  Explicitly, a pattern of sets on a topological space $X$
is a covariant functor $F\colon \Cl(X)^\op\to {\bf Set}$ to the category of sets
satisfying the
property that, for any given {\em finite} closed covering
$\{C_\lambda\}_{\lambda}$ of $X$, the canonical diagram
\begin{equation}
 F(X) \to \prod_{\lambda} F(C_\lambda)\rightrightarrows
\prod_{\lambda,\mu} F(C_\lambda\cap C_\mu)
\end{equation}
 is an equalizer diagram.
A pattern $F$ on a topological space is called {\em global} if for any
inclusion of closed sets $C'\subseteq C$ the restriction morphism
$F(C)\to F(C')$ is an epimorphism.

 We would like to note
that for compact Hausdorff spaces Leray's definition of
{\em faisceau continu} is equivalent to the
definition of a sheaf.  However, in this
paper we only consider sheaves over Alexandrov spaces which are of
completely different nature, and thus we cannot exchange these two
concepts.  On the other hand, for any finite
Alexandrov space, we show below
that the category of global patterns and the category of flabby
sheaves are equivalent up to a natural duality.

It follows from Lemma~\ref{FlipTopology}
that the lattice of open sets of an Alexandrov space
$(P,\leq)$ is isomorphic to the lattice of closed sets of the dual
Alexandrov space $(P,\leq)^\op$.  Hence:
\begin{prop}\label{PatternsEquivalentToSheaves}
  Let $(P,\leq)$ be a finite preordered set.  The category of (flabby)
  sheaves on an Alexandrov space $(P,\leq)$ is isomorphic to the
  category of (global) patterns on the opposite Alexandrov space
  $(P,\leq)^\op$.
\end{prop}
\begin{proof}
  Since the lattice of closed subsets of $(P,\leq)^\op$ is isomorphic
  to the lattice of open subsets of $(P,\leq)$, we
  conclude that any (flabby) sheaf on $(P,\leq)$ is a (global) pattern
  on $(P,\leq)^\op$ regardless of $P$ being finite.  Conversely,
  assume $F$ is a (global) pattern on $(P,\leq)^\op$, and let $\C{U}$
  be an open cover of $(P,\leq)$.  As $P$ is finite, the number of
  open and closed subsets of $P$ is finite as well.  Thus $\C{U}$ is a
  finite collection of closed subsets of $(P,\leq)^\op$  covering $P$.
  Furthermore,
  \begin{equation}\label{Equalizer}
    F(P) \to \prod_{U\in\C{U}} F(U)\rightrightarrows
    \prod_{U,U'\in\C{U}} F(U\cap U')
  \end{equation}
  is an equalizer diagram because  $F$ is a (global) pattern on $(P,\leq)^\op$.  
Hence  $F$ is a sheaf.
\end{proof}
\noindent
The restriction that $P$ is finite comes from the definition of a
pattern.  A pattern is a sheaf-like object where
Diagram~\eqref{Equalizer} is an equalizer for only {\em finite} closed
coverings, as opposed to a sheaf where Diagram~\eqref{Equalizer} is an
equalizer for every (finite or infinite) open covering.

 Next, we consider a poset $(P,\leq)$ as a category by letting
\begin{equation}
Ob(P)=P \quad \text{and}\quad
\Hom_P(p,q)=\left\{
\begin{array}{cl}
\{p\to q\} & \text{if}\ p\leq q,\\
  \emptyset & \text{otherwise}.
\end{array}\right.
\label{PosetCategory}
\end{equation}
Then a covariant
  functor $X\colon P\to \text{\bf Vect}_k$ 
to the category of vector spaces over $k$ is just a collection of 
vector spaces
  $\{X_p\}_{p\in P}$ together with linear maps
  $T_{qp}\colon X_p\to X_q$ such that (i) $T_{pp} = \id_{X_p}$ and
  (ii) $T_{rq}\circ T_{qp} = T_{rp}$.  
Any such a covariant functor will be called
   a right {\em $P$-module}.  The category of right $P$-modules and
  their morphisms will be denoted by $\lmod_{P}$. We will call a
  $P$-module flabby if each $T_{pq}$ is an epimorphism.  If $X\colon
  P\to {\bf Alg}_k$ is a functor into the category of $k$-algebras, then
  it will be referred as a right {\em $P$-algebra}.  
The category of right
  $P$-algebras and their morphisms will be denoted by ${\bf Alg}_P$.

For a topological space $X$ and a covering $\C{O}$ of $X$, we  say that
$\C{O}$ is stable under finite intersections if for any finite
collection $O_1,\ldots,O_n$ of sets from $\C{O}$ there exists a subset
$\C{O}'\subseteq \C{O}$ such that
\begin{equation}
 \bigcap_{i=1}^n O_i = \bigcup_{O'\in\C{O}'} O'.
\end{equation}

\begin{lemma}\label{LimitForm}
  Let $F$ be a sheaf of algebras on a topological space $X$.  Then for
  any open subset $U\subseteq X$ and any open covering $\C{U}$ of $U$
  that is stable under finite intersections, the canonical morphism
  $F(U)\to \lim_{V\in\C{U}} F(V)$ is an isomorphism.
\end{lemma}
\begin{proof}
  First, we recall that a pre-sheaf $F$ is a sheaf  on a topological
  space $X$ if and only if given an open subset $U$ and a covering
  $\C{U}$ of $U$ we have:
  \begin{enumerate}
  \item  for any $s\in F(U)$, if
 ${\rm Res}^U_V(s) = 0$ for all $V\in\C{U}$, then
$s=0$;
  \item for a collection of elements $\{s_V\in F(V)\}_{V\in\C{U}}$
    indexed by $\C{U}$ and  satisfying ${\rm Res}^V_{V\cap W}(s_V) = {\rm
      Res}^W_{V\cap W}(s_W)$, there exists $s\in F(U)$ with ${\rm
      Res}^U_V(s) = s_V$ for any $V\in\C{U}$.
  \end{enumerate}
  Now let $F$ be a sheaf and $\C{U}$ be an open covering of an
  open subset $U$. Assume that the covering $\C{U}$ is stable under finite intersections.  
Next, recall
  that
  \begin{equation}\label{LimitOnACovering}
    \lim_{V\in\C{U}} F(V) = \{(f_V)_{V\in\C{U}}\;|\;f_V\in F(V)\text{
    and } f_W = {\rm Res}^V_W(f_V) \text{ for any } V\supseteq
  W\in\C{U}\}.
  \end{equation}
  The canonical morphism $F(U)\to \lim_{V\in\C{U}} F(U)$ sends each
  element $s\in F(U)$ to the sequence $({\rm
    Res}^U_V(s))_{V\in\C{U}}$.  The condition $(1)$ implies that the
  canonical morphism is injective.  Every element $(f_V)_{V\in\C{U}}$
  of $\lim_{V\in\C{U}} F(V)$ satisfies ${\rm Res}^V_{V\cap W}(f_V) =
  {\rm Res}^W_{V\cap W}(f_W)$ because of the fact that $\C{U}$ is
  stable under finite intersections and because $F$ is a sheaf.  Then
  the condition $(2)$ implies that the canonical morphism is an
  epimorphism.
\end{proof}

The following result is well known for sheaves of modules.
See~\cite[Prop.~6.6]{BrunBrunsRomer:CohomologyOfPosets} for a proof.
Here we prove an analogous result for sheaves of algebras.
\begin{thm}\label{PAlgebras}
  Let $(P,\leq)$ be a poset.  Then the category of sheaves of
  $k$-algebras on $P$ with the Alexandrov topology induced by the poset structure is equivalent to the
  category $P$-algebras.
\end{thm}
\begin{proof}
  Consider an arbitrary sheaf of algebras $F\in \bs{Sh}(P)$.  Define a
  collection of $k$-algebras $\C{F}:=\{F_p\}_{p\in P}$ indexed by elements of
  $P$ by letting $F_p:= F(\upset p)$ for any $p\in P$.  Then
 $\upset p\supseteq \upset q$ for any
  $p\leq q$. Therefore, since $F$ is a sheaf, we have
   morphisms of $k$-algebras $T^F_{qp}\colon F_p\to F_q$
  for any $p\leq q$ that satisfy (i) $T^F_{pp}= \id_{F_p}$ for any
  $p\in P$, and (ii) $T^F_{rq}\circ T^F_{qp} = T^F_{rp}$ for any $p\leq
  q\leq r$ in $P$.  In other words,
  $\{F_p\}_{p\in P}$ is a right $P$-algebra.  Also, given any morphism
  of sheaves $\phi\colon F\to G$, we have morphisms of
  algebras $\{\phi_{\upset\; p}=:\varphi_p\}_{p\in P}$ that fit into a commutative diagram of
  the form
  \begin{equation}
\xymatrix{
    F_p \ar[d]_{\varphi_p} \ar[r]^{T^F_{qp}} & F_q \ar[d]^{\varphi_q}\\
    G_p \ar[r]_{T^G_{qp}} &\,G_q\,.}
  \end{equation}
  This means that we have a functor of the form $\Phi\colon \bs{Sh}(P)\to
  {\bf Alg}_P$.

Conversely, assume that we have such a collection of
  algebras $\C{F} = \{F_p\}_{p\in P}$ with structure morphisms
  $T_{qp}\colon F_p\to F_q$ for any $p\leq q$ satisfying the
  conditions (i) and (ii) described above.  We let $\Upsilon(\C{F})(V)
  := \lim_{v\in V} F_v$ viewing $P$ as a category as  in
  \eqref{PosetCategory}.  Now, for any inclusion of open sets
  $V\subseteq W$, we have a morphism of algebras
  ${\rm Res}^W_V(\Upsilon(\C{F}))\colon \Upsilon(\C{F})(W)\to
  \Upsilon(\C{F})(V)$. By definition, it is the canonical
  morphism of limits
  $
 \lim_{w\in W} F_w \to \lim_{v\in V} F_v.
$
 Thus we see that
  $\Upsilon(\C{F})$ is a pre-sheaf.

To show that
  $\Upsilon(\C{F})$ is a sheaf, we fix an open subset
  $U\subseteq P$ and an open covering $\C{U}$ of $U$. Using the
  description analogous to Equation~\eqref{LimitOnACovering}, one can
  see that for any 
\begin{equation}
f^U:=(f_u)_{u\in U}\in \Upsilon(\C{F})(U) = \lim_{u\in
    U} F_u
\end{equation}
 we have ${\rm Res}^U_V(f^U) = 0$ for any $V\in\C{U}$ if and
  only if $f_u=0$ for any $u\in U$.  Moreover, assume that we have a
  collection of elements $f^V:= (f^V_v)_{v\in V}\in
  \Upsilon(\C{F})(V)$ indexed by $V\in\C{U}$ satisfying ${\rm
    Res}^V_{V\cap W}\Upsilon(\C{F})(f^V) = {\rm Res}^W_{V\cap
    W}\Upsilon(\C{F})(f^W)$ for any $V,W\in\C{U}$.  This means $f^V_z
  = f^W_z$ for any $z\in V\cap W$.  Notice also that
  $ {\rm Res}^V_Z((f^V_v)_{v\in V}) = (f^V_z)_{z\in Z}\in
  \Upsilon(\C{F})(Z) $ for any open subset $Z$ of~$V$.  Therefore,
  we can patch $\{(f^V_v)_{v\in V}\}_{V\in\C{U}}$ by forgetting the superscripts indicating which
  open subset we consider and letting
  $f=(f_u)_{u\in U}$.  Hence we  conclude that $\Upsilon(\C{F})$
   is a sheaf.

Next, to show that $\Upsilon$ is compatible
  with the morphisms, for an arbitrary morphism\linebreak
 \mbox{$ \varphi \colon \{F_p\}_{p\in
    P}\to \{G_p\}_{p\in P}$} of right $P$-algebras and for any open
  subset $V\subseteq P$, we define
  \begin{equation}
 \Upsilon(\varphi)(V):= \lim_{v\in V}
  \varphi_v\colon \Upsilon(\{F_p\}_{p\in P})(V)\longrightarrow
  \Upsilon(\{G_p\}_{p\in P})(V).
\end{equation}
 Thus we obtain a functor from the category of $P$-algebras
  into the category of sheaves of $k$-algebras on~$P$. 

Finally, we verify that the functors $\Phi$ and $\Upsilon$ establish an equivalence of
categories. One can
  see that $\Upsilon(\Phi(F))(V) = \lim_{v\in V} F(\upset v)$.  Since
  $\{\upset v|\ v\in V\}$ is an open cover of $V$ that is stable
under finite intersections and  $F$ is a
  sheaf, it follows from Lemma~\ref{LimitForm} (cf.~\cite[Sect.~2.2,
  p.~85]{KashiwaraSchapira:SheavesOnManifolds})
that $F(V) \cong \lim_{v\in V} F(\upset v)$.
Hence we conclude that the endofunctors
  $\id$ and $\Upsilon\circ\Phi$ are isomorphic. 
On the other hand, since any $p\in P$ is
the unique minimal element of the open set
  $\upset p$, we obtain $F_p\cong\lim_{q\in\upset\;p}F_q$, so
  that $\Phi\circ\Upsilon$ is isomorphic to the identity functor.  
\end{proof}

\section{Classification of finite coverings via the universal partition
space $\pinf$}\label{Coverings}

The aim of this section is to establish an equivalence between
the category of finite coverings of algebras and an appropriate
category of finitely-supported flabby sheaves of algebras. To this end,
we first define a number of different categories of coverings and
 sheaves.
Then we explore their interrelations to assemble a path of functors
yielding the desired equivalence of categories.

\vspace{1.5mm}
\subsection{Categories of coverings}
\label{CategoryOfCoverings}~

\noindent
Let $X$ be a topological space and $\C{C}$ be a collection of subsets
of $X$ that cover $X$, i.e., $\bigcup_{U\in\C{C}} U = X$.  We allow
$\emptyset\in\C{C}$.  Recall that such a set is called {\em a
  covering} of $X$.  A covering $\C{C}$ is called finite if the set
$\C{C}$ is finite.  A covering $\C{C}$ of a topological space $X$ is
called {\em closed} (resp. {\em open}) if $\C{C}$ consists of closed
(resp. open) subsets of $X$.  Let us now 
 consider the category of pairs of
the form $(X,\C{C})$ where $X$ is a topological space and $\C{C}$ is a
closed (or open) covering.  A morphism $f\colon (X,\C{C})\to
(X',\C{C}')$ is a continuous map of topological spaces $f\colon X\to
X'$ such that for any $C\in\C{C}$ there exists $C'\in\C{C}'$ with the
property that $C\subseteq f^{-1}(C')$.  In the spirit of the Gelfand
transform, we are going to dualize this category to the category of
algebras.

Let $\Pi=\{\pi_i\colon A\to A_i\}_i$ be a finite set of epimorphisms
of algebras.  We allow the case $A_i=\bs{0}$ for some $i$.  Denote by
$\Lambda$ the lattice of ideals generated by $\ker\pi_i$, where $\cap$
and $+$ denote the join and meet operations respectively.  Recall
from \cite{HKMZ:PiecewisePrincipalComoduleAlgebras} that the set $\Pi$
is called a {\em covering} if the lattice $\Lambda$ is distributive
and $\bigcap_i \ker(\pi_i)=\bs{0}$.  Finally, an ordered family
$\ord{\Pi} = (\pi_i\colon A\rightarrow A_i)_i$ is called an {\em
  ordered covering} if the set $\kappa(\ord{\Pi}):= \{\pi_i\colon
A\rightarrow A_i\}_i$ is a covering.  In such an ordered sequence
$(\pi_i\colon A\to A_i)_i$ we allow repetitions.

In~\cite{HKMZ:PiecewisePrincipalComoduleAlgebras},
for each natural number $N$,
the authors defined
a category $\ncov{N}$  whose objects
are pairs $(A;\pi_0,\ldots,\pi_N)$, where $A$ is a unital algebra
and the ordered sequence $(\pi_0,\ldots,\pi_N)$ is an ordered covering
of~$A$. (Note that herein we begin labelling covering elements from 
$0$ rather than from $1$ as 
in~\cite{HKMZ:PiecewisePrincipalComoduleAlgebras}.)   
A morphism between two objects $f\colon
(A;\pi_0,\ldots,\pi_N)\to(A';\pi_0',\ldots,\pi_N')$ is a morphism of
algebras $f\colon A\to A'$ such that $f(\ker(\pi_i))\subseteq
\ker(\pi_i')$ or, equivalently, such that $\ker(\pi_i)\subseteq
f^{-1}(\ker(\pi_i'))$, for any $i=0,\ldots,N$.  This category is called
{\em the category of ordered $(N+1)$-coverings of algebras}.

For any natural number $N$, there is a functor $e_N\colon
\ncov{N}\to\ncov{N+1}$  defined on
the set of objects as
$e_N(A;\pi_0,\ldots,\pi_N):= (A;\pi_0,\ldots,\pi_N, A\to \bs{0})$ for 
all $(A;\pi_0,\ldots,\pi_N)\in Ob(\ncov{N})$, and 
as  identity on the sets of morphisms. Thus $e_N$ is a faithful
functor. It is also full because, for any  $(A,\ord{\Pi})$ and
$(A',\ord{\Pi}')$ in $Ob(\ncov{N})$,  we have
\begin{equation}
 \Hom_{\ncov{N+1}}(e_N(A,\ord{\Pi}),e_N(A',\ord{\Pi}')) =
\Hom_{\ncov{N}}((A,\ord{\Pi}), (A',\ord{\Pi}')).
\end{equation}

Our next step is to introduce 
the category $\ocovf$ of pairs of the form
$(A,\ord{\Pi})$ where $A$ is again a unital algebra but
$\ord{\Pi}$ is 
an infinite (rather than finite) 
sequence of epimorphisms $\pi_i\colon A\to
A_i$, $i\in\B{N}$, such that: (i) all but
finitely many of these epimorphisms have zero codomain and (ii) the
underlying set $\kappa(\ord{\Pi})$ of epimorphisms is a covering 
of~$A$.  A morphism $f\colon (A; \pi_0, \pi_1, \ldots)\to (A'; \pi_0',
\pi_1', \ldots)$ is a morphism of algebras $f\colon A\to A'$ with the
property that $\ker(\pi_i)\subseteq f^{-1}(\ker(\pi_i'))$ for any
$i\in\B{N}$. Alternatively, we can define $\ocovf$ as a colimit:
\begin{dfn}\label{Def:OrderedCoverings}
 The category $\ocovf:={\rm colim}_{N\in\mN} \ncov{N}$ is called
  the category of {\em finite ordered coverings} of algebras.
\end{dfn}

 Next,  recall from the
beginning of this section that, in the category of
topological spaces together with a prescribed finite covering,
  a covering is a collection
of sets devoid of an ordering on the covering sets.  Thus, it is
necessary for us to replace the ordered sequences of epimorphisms in
the objects of the category $\ocovf$, and work with {\em finite sets}
of epimorphisms of algebras.
\begin{dfn}\label{Def:Coverings}
  Let $\cov$ be a category whose objects are pairs $(A,\Pi)$, where
  $A$ is a unital algebra and $\Pi$ is a finite {\em set} of unital
  algebra epimorphisms that is a covering of the algebra~$A$.  A
  morphism $f\colon (A,\Pi)\to (A',\Pi')$ in this category is a
  morphism of algebras $f\colon A\to A'$ satisfying the condition that
  for any epimorphism $\pi'_i\colon A'\to A'_i$ in the covering $\Pi'$
  there exists an epimorphism $\pi_j\colon A\to A_j$ in the covering
  $\Pi$ such that $\ker(\pi_j)\subseteq f^{-1}(\ker(\pi'_i))$.  This
  category will be called {\em the category of finite coverings of
    algebras}.
\end{dfn}
\noindent
 If $f\colon (A,\Pi)\to (A',\Pi')$ is a
morphism in $\cov$, we will say that $f$ is implemented by the morphism
of algebras $f\colon A\to A'$.  Note that the matching of the
epimorphisms, or rather the kernels, is not part of the datum defining
a morphism.   

We also need the
following auxiliary category.
\begin{dfn}
   Category $\covaux$ is a category whose
  objects are the same as the objects of $\ocovf$. A morphism
  $f\colon (A,\ord{\Pi})\to (A',\ord{\Pi}')$ in $\covaux$
is a morphism of algebras
  $f\colon A\to A'$  satisfying the property that for every
  $\pi_j'$ appearing in the sequence $\ord{\Pi}'$ there exists an
  epimorphism $\pi_i$ appearing in the ordered sequence $\ord{\Pi}$
  such that $\ker(\pi_i)\subseteq f^{-1}(\ker(\pi_j'))$.
\end{dfn}
\noindent
As before, the matching of the epimorphisms is not part of the datum
defining a morphism.

Now we want to prove that the categories $\covaux$ and $\cov$ are
equivalent.
Recall first that a
 functor $F\colon \C{C}\to\C{D}$ is called {\em essentially
  surjective} if for every $X\in Ob(\C{D})$ there exists an object
$C_X\in Ob(\C{C})$ and an isomorphism $\omega_X\colon F(C_X)\to X$ in
$\C{D}$.
\begin{thm}{\cite[IV.~4 Thm.1]{MacLane:Categories}}\label{MainLemma}
  Let $F\colon \C{C}\to \C{D}$ be a functor that is fully faithful
  and essentially surjective.  Then $F$ is an equivalence of
  categories.
\end{thm}

\begin{lemma}\label{Replacement}
 For every object $(A;\pi_0,\pi_1,\ldots)$ and morphism $f\colon
  (A,\ord{\Pi})\to (A',\ord{\Pi}')$ in the category $\covaux$ 
consider the assignment
  \[ \G{Z}(A;\pi_0,\pi_1,\ldots) := (A; \{\pi_i|\ i\in\B{N}\})\in 
Ob(\cov)
  \quad\text{ and }\quad \G{Z}(f) := f\in\text{\it Mor}(\cov).
  \] 
  The assignment defines a functor $\G{Z}\colon \covaux\to\cov$
  establishing the equivalence of categories.
\end{lemma}
\begin{proof}
  One can see that
  \begin{equation}
  \Hom_{\covaux}((A,\ord{\Pi}),\ (B,\ord{\Theta}))
  = \Hom_{\cov}((A,\kappa(\ord{\Pi})),\ (B,\kappa(\ord{\Theta}))).
  \end{equation}
  This implies that $\G{Z}$ is fully faithful,
  and that it makes sense for the
  functor $\G{Z}$ to act as identity on the set of morphisms.
  Given an object $(A,\Pi)$ in $\cov$, one can choose an ordering on
  the finite set $\Pi$ and obtain an ordered sequence of epimorphisms
\begin{equation}
(\pi_0\colon A\to A_0, \pi_1\colon A\to A_1,\ldots,\pi_N\colon A\to
  A_N),
\end{equation}
 where $N=|\Pi|-1$.  We can pad this sequence with $A\to \bs{0}$
  to get an infinite sequence $\ord{\Pi}$ of epimorphisms where only
  finitely many epimorphisms are non-trivial.  This infinite sequence
  has the property that the corresponding finite set
  $\kappa(\ord{\Pi})$ of epimorphisms is the set
  $\Pi\cup\{A\to\bs{0}\}$.  Since the identity morphism $\id_A\colon
  A\to A$ implements an isomorphism
\begin{equation}
(A,\Pi\cup\{A\to\bs{0}\})\rightarrow
  (A,\Pi)
\end{equation}
 in $\cov$, we conclude that $\G{Z}$ is essentially
  surjective. Now the result follows from Theorem~\ref{MainLemma}.
\end{proof}

The
category $\covaux$ sits in between the category $\ocovf$ of ordered
coverings and the category $\cov$ of coverings:
\begin{equation}
 \ocovf \hookrightarrow \covaux \xra{\simeq} \cov.
\end{equation}
The definitions of morphisms in the categories $\covaux$ and $\cov$
coincide even though the classes of objects are different.  On the other
 hand, the categories $\ocovf$ and $\covaux$ share the same objects, but
there are more morphisms in $\covaux$ than in $\ocovf$:
\begin{equation}
  \text{Hom}_{\ocovf}((A,\ord{\Pi}),(B,\ord{\Pi}'))\subseteq
\text{Hom}_{\covaux}((A,\ord{\Pi}),(B,\ord{\Pi}')).
\end{equation}
Explicitly, one can describe
$\text{Hom}_{\covaux}((A,\ord{\Pi}),(B,\ord{\Pi}'))$ as the set of
morphisms of algebras
{$f\colon A\rightarrow B$} for which there exists
a sequence of epimorphisms $\ord{\Pi}''$ obtained from $\ord{\Pi}$ by
permutations and insertions of already existing epimorphisms, and such
that $f$ is a morphism in
$\text{Hom}_{\ocovf}((A,\ord{\Pi}'')$, $(B,\ord{\Pi}'))$.
This can be elegantly expressed  by
introducing another auxiliary category $\auxp$ such that $\covaux$ comes
out as the quotient of $\auxp$ by an equivalence relation on the
morphisms (c.f. Definition~\ref{Aux2} and Lemma~\ref{eqauxp} below).

The reason why we prefer working with ordered sequences of
epimorphisms in $\covaux$ rather then the sets of epimorphisms in $\cov$
is that we want to interpret coverings in the language of sheaves.
Working with sheaves inevitably introduces order on the set of
epimorphisms because of the particular nature of morphism in the
category of sheaves (c.f.~Lemma~\ref{Equivalence}).  Fortunately, by
Lemma~\ref{Replacement}, our auxiliary category $\covaux$, where
the objects are based on ordered sequences, is equivalent to $\cov$,
the category of finite coverings of algebras where the objects are
based on finite sets of epimorphisms.

Let $\alpha:\mN\rightarrow\mN$ be a tame surjection from the monoid
$\mathcal{M}$ (Definition~\ref{FiniteFibers}).  Any such $\alpha$
gives rise to an endofunctor $\check{\alpha}\colon
\ocovf\rightarrow\ocovf$ defined on objects by
\begin{equation}
  \check{\alpha}(A,(\pi_i)_i):=(A,(\pi_{\alpha(i)})_i),
\end{equation}
and by identity on the morphisms.
\begin{dfn}\label{Aux2}
 The category $\auxp$ is a category whose objects are the same as in
  $\ocovf$ and $\covaux$, and whose  morphisms  are pairs 
  $(f,\alpha):(A,\ord{\Pi})\rightarrow(A',\ord{\Pi}')$ such that
  $\alpha\in\mathcal{M}$ and
\begin{equation*}
  f:\check{\alpha}(A,\ord{\Pi})\longrightarrow (A',\ord{\Pi}')
\end{equation*}
is a morphism in $\ocovf$. The identity morphisms are simply
$(\id_A,\id_\mN)$, and the composition of morphisms is defined as
\begin{equation*}
(g,\beta)\circ (f,\alpha)=(g\circ(\check{\beta}f),\alpha\circ\beta).
\end{equation*}
\end{dfn}
\noindent Note that we have
$(\beta\circ\alpha)^{\textstyle\check{}}=\check{\alpha}\check{\beta}$.

We define an equivalence relation on $\auxp$ as follows.  We say that
two morphisms $(f,\alpha)$, $(f',\alpha')$ in
$\text{Hom}_{\auxp}((A,\ord{\Pi}),(A',\ord{\Pi}'))$ are equivalent
(here denoted by $(f,\alpha)\sim(f',\alpha')$) if $f=f'$ as morphisms
of algebras. By \cite[Proposition~II.8.1]{MacLane:Categories}, we know
the quotient category $\auxp/\!\!\sim$ exists.
Moreover, it is easy to see that the relation $\sim$
preserves the compositions of morphisms.  Hence, by the proof of
\cite[Proposition~II.8.1]{MacLane:Categories}, we do not need to extend
the relation $\sim$ to form a quotient category.
We are now ready  for:
\begin{lemma}\label{eqauxp}
  The category $\covaux$ and the quotient category $\auxp/\!\!\sim$ are
  isomorphic.
\end{lemma}
\begin{proof}
  We implement the isomorphism with two functors
  \begin{equation}
    F:\auxp/\!\!\sim\ \longrightarrow\covaux,\quad G:\covaux\longrightarrow\auxp/\!\!\sim\,,
  \end{equation}
  defined as identities on objects. For any equivalence class
  $[f,\alpha]_\sim$ of morphisms in $\auxp/\!\!\sim$, we define
  $F([f,\alpha]_\sim):=f$. On the other hand, for any morphism
  $f:(A,({\pi}_i)_{i\in\mN})\rightarrow (A',({\pi'}_i)_{i\in\mN})$
in $\covaux$, we set
  $G(f):=[f,\alpha]_\sim$, where $\alpha$ is any
element of $\mathcal{M}$ satisfying:
\begin{equation}\label{alphaf}
\alpha(i)=
\begin{cases}
 i-N  & \text{ for } i> N,\\
 j, \text{ where $j$ is such that }
  \ker\pi_j\subseteq f^{-1}(\ker{\pi'}_i), & \text{ for } i\leq N.
\end{cases}
\end{equation}
Here $N\in\mN$ is a number such that for any
  $i>N$ we have ${\pi'}_i:=A'\rightarrow 0$.
It is obvious that $F\circ G=\id_{\covaux}$ and $G\circ
F=\id_{\auxp/\!\!\sim}$. One can easily see that $F$ and $G$ are
functorial --- it is enough to note that $\check{\alpha}f=f$ as morphisms
of algebras.
\end{proof}
\vspace{1.5mm}

\subsection{The sheaf picture for coverings}
\label{SheavesAndCoverings}~

\noindent
Let $\Sh$ be the category of flabby sheaves of algebras over
$\pinf$.  A morphism $f\colon F\to G$ in $\Sh$ is a collection
$\{f_U\colon F(U)\to G(U)\}_{U\in\Top(\pinf)}$ of morphisms of
algebras (indexed by the open subsets of $\pinf$)
 that fit into the following commutative diagram
\begin{equation}
\xymatrix{
  \ar[d]_{{\rm Res}^U_V(F)} F(U) \ar[r]^{f_U} & G(U) \ar[d]^{{\rm Res}^U_V(G)} \\
   F(V) \ar[r]_{f_V} & G(V)
}
\end{equation}
for any chain $V\subseteq U$ of open subsets of~$\pinf$.

\begin{dfn}\label{FlabbySheavesWithFiniteSupport}
  A  flabby sheaf $F\in Ob(\Sh)$ is said to have
  finite support if there exists $N\in\B{N}$ such that $F(\B{A}_n) =
  0$ for any $n>N$.  The full subcategory of flabby sheaves with
  finite support will be denoted by $\fSh$.
\end{dfn}

Here is an alternative way of seeing sheaves
with finite support on $\pinf$.  Any sheaf of algebras on $\pN$ can be
extended to a sheaf of algebras on $\B{P}^{N+1}(\ztwo)$ by the direct image functor
\begin{equation}
{\rm Sh}(\pN)\ni F\longmapsto(\phi_N)_*(F)\in{\rm Sh}(\B{P}^{N+1}(\ztwo))
\end{equation}
with respect to the canonical
embedding
$\phi_N\colon \pN\to \B{P}^{N+1}(\ztwo)$ defined in Lemma~\ref{Canonical}.  Then we obtain an
injective system of %small 
categories $({\rm Sh}(\pN),j_N)$ whose colimit can be identified with
 $\fSh$.

For a flabby sheaf $F$ in $Ob(\fSh)$, we will use ${\rm Res}_i(F)$ to
denote the canonical restriction epimorphism $F(\pinf)\to F(\B{A}_i)$
for any $i\in\B{N}$.  Note that, since $F$ is a sheaf with finite
support, all but finitely many morphisms ${\rm Res}_i(F)$ are of the
form $F(\pinf)\to \bs{0}$.
 The following lemma is a reformulation of
\cite[Cor.~4.3]{HKMZ:PiecewisePrincipalComoduleAlgebras} in a new
setting.  (Cf.~\cite[Prop.~1.10]{Yuzvinski:CMRingsOfSections} for a commutative version.)
The proof uses Lemma~\ref{LatticeMorphism} and 
is essentially the same as in
\cite[Prop.~2.2]{HKMZ:PiecewisePrincipalComoduleAlgebras}.
  Note that we can apply the generalized
Chinese
Remainder Theorem (e.g., see \cite[Thm.~18 on p.~280]{SamuelZariski} and
\cite{Maszczyk:Patterns}) as there is always only a finite
number of non-trivial congruences.
\begin{lemma}\label{Equivalence}
  For any $(A,{\ord{\Pi}})\in Ob(\ocovf)$ and $F\in\fSh$, the
  following assignments
  \begin{align}
    \label{ggb}
    \Psi(A,{\ord{\Pi}}) := & \left\{U \mapsto A/R^{\ord{\Pi}}(U)\right\}_{U\in\Top(\pinf)}\in \fSh,\\
    \label{gga}
    \Phi(F) := & \left(F(\pinf);\ {\rm Res}_0(F),\ {\rm Res}_1(F),\
      \ldots,\ {\rm Res}_n(F),\ \ldots\right) \in \ocovf,
  \end{align}
  yield functors establishing an equivalence between
  the category $\ocovf$ of ordered coverings and  the category $\fSh$
  of finitely-supported flabby sheaves of algebras over $\pinf$.
\end{lemma}

We would like to extend the equivalence we constructed in
Lemma~\ref{Equivalence} to an equivalence of categories between
$\covaux$ (and therefore $\cov$) and a suitable category of sheaves
filling the following diagram:
\begin{equation}\label{diagramcat}
\xymatrix{
  \ar[d]_\Psi^\simeq \ocovf \ar[r]  &
  \covaux \ar@{..>}[d]^\simeq \ar[r]^-{ \G{Z} }_{\simeq} & \cov \\
  \fSh  \ar@{..>}[r] & \pfSh.
}
\end{equation}
As $\covaux$ is isomorphic to a quotient
category, we expect $\pfSh$ to be a quotient of
the following category of sheaves with extended morphisms:
\begin{dfn}
  The objects of $\pfShp$ are
finitely-supported flabby sheaves of algebras over $\pinf$.
  A morphism $[\tilde{f},\alpha^*]:P\rightarrow Q$ in
  $\pfShp$ is a pair consisting of a continuous map
  (see~\eqref{fSigma})
\begin{equation*}
\alpha^*:\pinf\longrightarrow\pinf,\quad \chi_a\longmapsto\chi_{\alpha^{-1}(a)},
\end{equation*}
where $\mathcal{M}\ni\alpha:\mN\rightarrow\mN$
is a tame surjection (Definition~\ref{FiniteFibers}),
and a morphism of sheaves
\begin{equation*}
  \tilde{f}:\alpha^*_*P\rightarrow Q.
\end{equation*}
Composition of morphisms is given by
\begin{equation*}
  [\tilde{g},\beta^*]\circ[\tilde{f},\alpha^*]:=[\tilde{g}\circ (\beta^*_* \tilde{f}),\beta^*\circ\alpha^*].
\end{equation*}
\end{dfn}

To define $\pfSh$ as a quotient category 
equivalent to $\covaux\cong\auxp/\!\!\sim$, 
we proceed by first
proving the equivalence of $\pfShp$ and  $\auxp$.
\begin{lemma}
Let $\Psi:\ocovf\rightarrow\fSh$ and $\Phi:\fSh\rightarrow\ocovf$ be
functors defined in
Lemma~\ref{Equivalence}. Then the functors
\begin{equation*}
\pPsi:\auxp\longrightarrow\pfShp,\quad \pPhi:\pfShp\longrightarrow\auxp,
\end{equation*}
defined on objects by
\begin{equation*}
\pPsi(A,\ord{\Pi})=\Psi(A,\ord{\Pi}),\quad \pPhi(P)=\Phi(P),
\end{equation*}
and on morphisms by
\begin{equation*}
\pPsi(f,\alpha)=[\Psi f,\alpha^*],\quad \pPhi[\tilde{f},\alpha^*]=(\Phi\tilde{f},\alpha),
\end{equation*}
establish an equivalence of categories between $\auxp$ and $\pfShp$.
\end{lemma}
\begin{proof}
 We divide the proof
  into several steps.
\begin{enumerate}
\item $(\alpha^*)^{-1}(\B{A}_i)=\B{A}_{\alpha(i)}$ for all $i\in\mN$.
Indeed,
\begin{align*}
(\alpha^*)^{-1}(\B{A}_i)&=(\alpha^*)^{-1}(\{\chi_a\;|\;i\in a\subset\mN\})\nonumber\\
&=\{\chi_b\;|\;\alpha^*(\chi_b)=\chi_a \text{\ and\ }i\in a\subset\mN\}\nonumber\\
&=\{\chi_b\;|\;\chi_{\alpha^{-1}(b)}=\chi_a \text{\ and\ }i\in a\subset\mN\}\nonumber\\
&=\{\chi_b\;|\;i\in\alpha^{-1}(b)\}\nonumber\\
&=\{\chi_b\;|\;\alpha(i)\in b\subset\mN\}\nonumber\\
&=\B{A}_{\alpha(i)}.
\end{align*}
\item As $\alpha$ is tame by assumption, $\alpha^{-1}(a)$ is finite
  for any finite $a\subseteq\mN$. Hence $\alpha^*$ is well defined.
\item Equality $\alpha^*=\beta^*$ implies $\alpha=\beta$ for any
  surjective maps $\alpha,\beta:\mN\rightarrow\mN$.  Hence the functor
  $\pPhi$ is well defined.
\item $\alpha^*_*\Psi=\Psi\check{\alpha}$. Indeed, for any
  $(A,(\pi_i)_i)\in\auxp$, we see that
\begin{align*}
(\alpha^*_*\Psi)((A,(\pi_i)_i))&=\alpha^*_*(U\mapsto A/R^{(\pi_i)_i}(U))\nonumber\\
&=U\mapsto A/R^{(\pi_i)_i}((\alpha^*)^{-1}(U)),\nonumber\\
(\Psi\check{\alpha})((A,(\pi_i)_i))&=\pPsi((A,(\pi_{\alpha(i)})_i))\nonumber\\
&=U\mapsto A/R^{(\pi_{\alpha(i)})_i}(U).
\end{align*}
On the other hand,
the observation that for any open $U\subseteq\pinf$ we have
 $ U=\bigcup_{a \text{\ s.t.}\atop \chi_a\in U}\bigcap_{i\in a}\B{A}_i$
and the result from Step~(1), yield:
\begin{align*}
  R^{(\pi_i)_i}((\alpha^*)^{-1}(U))&=
  R^{(\pi_i)_i}((\alpha^*)^{-1}(\bigcup_{a \text{\ s.t.}\atop \chi_a\in U}\bigcap_{i\in a}\B{A}_i))\nonumber\\
  &=R^{(\pi_i)_i}(\bigcup_{a \text{\ s.t.}\atop \chi_a\in U}\bigcap_{i\in a}(\alpha^*)^{-1}(\B{A}_i))\nonumber\\
  &=R^{(\pi_i)_i}(\bigcup_{a \text{\ s.t.}\atop \chi_a\in U}\bigcap_{i\in a}\B{A}_{\alpha(i)})\nonumber\\
  &=\bigcap_{a \text{\ s.t.}\atop \chi_a\in U}\left(\sum_{i\in a}\ker\pi_{\alpha(i)}\right)\nonumber\\
  &=R^{(\pi_{\alpha(i)})_i}(\bigcup_{a \text{\ s.t.}\atop \chi_a\in U}\bigcap_{i\in a}\B{A}_i)\nonumber\\
  &=R^{(\pi_{\alpha(i)})_i}(U).
\end{align*}
\item Let $\alpha,\beta:\mN\rightarrow\mN$ be maps from
  $\mathcal{M}$. Then
  $(\alpha\circ\beta)^*=\beta^*\circ\alpha^*$. Indeed, for any
  $\chi_a\in\pinf$, we obtain:
\begin{equation*}
  (\beta^*\circ\alpha^*)(\chi_a)=\beta^*(\chi_{\alpha^{-1}(a)})=\chi_{(\beta^{-1}\circ\alpha^{-1})(a)}
  =\chi_{(\alpha\circ\beta)^{-1}(a)}=(\alpha\circ\beta)^*(\chi_a).
\end{equation*}
\item $\pPsi$ is functorial. Indeed, take any composable morphisms
  $(f,\alpha)$ and $(g,\beta)$ in $\auxp$. Then the previous two steps
  and the functoriality of $\Psi$ yield
  \begin{align*}
    \pPsi((g,\beta)\circ(f,\alpha))&=\pPsi((g\circ(\check{\beta} f),\alpha\circ\beta))\nonumber\\
    &=(\Psi(g\circ(\check{\beta} f)),(\alpha\circ\beta)^*)\nonumber\\
    &=(\Psi(g)\circ\Psi(\check{\beta} f)),\beta^*\circ\alpha^*)\nonumber\\
    &=((\Psi g)\circ(\beta^*_*\Psi f),\beta^*\circ\alpha^*)\nonumber\\
    &=[\Psi g,\beta^*]\circ [\Psi f,\alpha^*]\nonumber\\
    &=\pPsi((g,\beta))\circ\pPsi((f,\alpha)).
  \end{align*}
\item $\Phi\alpha^*_*=\check{\alpha}\Phi$.  Indeed,
take any $P\in\pfShp$.
  Using the result of Step~(1), we obtain:
  \begin{align*}
    (\Phi\alpha^*_*)(P)&=\Phi(U\mapsto P(\alpha^{-1}(U)))\nonumber\\
    &=(P(\pinf),(P(\pinf)\mapsto P((\alpha^*)^{-1}(\B{A}_i)))_i)\nonumber\\
    &=(P(\pinf),(P(\pinf)\mapsto P(\B{A}_{\alpha(i)}))_i)\nonumber\\
    &=\check{\alpha}((P(\pinf),(P(\pinf)\mapsto P(\B{A}_{i}))_i))\nonumber\\
    &=(\check{\alpha}\Phi)(P).
  \end{align*}
\item $\pPhi$ is functorial. The proof uses the result from the
  previous step, and is analogous to the proof of  Step~(6).
\item The natural isomorphism $\eta\colon\Psi\Phi\rightarrow\id_{\fSh}$
  comes from a family of isomorphisms of sheaves $\eta_P\colon\Psi\Phi
  P\rightarrow P$. The latter are given by
the canonical isomorphisms between the  image of an epimorphism
 and the quotient of its domain by its kernel
  (c.f.~\cite[Prop.~2.2]{HKMZ:PiecewisePrincipalComoduleAlgebras}):
%\begin{equation*}
$\eta_{P,U}:P(\pinf)/\ker(P(\pinf)\rightarrow P(U))
\;\longrightarrow\; P(U)$.
%\end{equation*}
To see that
%\begin{equation*}
$\alpha^*_*\eta_P=\eta_{\alpha^*_*P}$
%\end{equation*}
for any sheaf~$P$,
note that $\alpha^*_*\eta_P:\alpha^*_*\Psi\Phi
P=\Psi\Phi\alpha^*_*P\longrightarrow \alpha^*_*P$ and
\begin{align*}
\ \qquad (\alpha^*_*\eta_P)_U
&:=\eta_{P,(\alpha^*)^{-1}(U)}\nonumber\\
  &\phantom{:}=P(\pinf)/\ker(P(\pinf)\rightarrow P((\alpha^*)^{-1}(U)))
\;\longrightarrow\; P((\alpha^*)^{-1}(U))\nonumber\\
  &\phantom{:}=\eta_{\alpha^*_*P,U}.
\end{align*}
Here the first equality is just the definition the of action of the 
direct image
functor on morphisms.
\item The family of maps
%\begin{equation*}
$\widetilde{\eta}_P:=[\eta_P,\id_{\mN}^*]:\pPsi\pPhi P\rightarrow P$
%\end{equation*}
establishes a natural isomorphism between $\pPsi\pPhi$ and
$\id_{\pfShp}$. It is clear that $\widetilde{\eta}_P$'s are
isomorphisms.  We know that $\eta$ is a natural isomorphism. In
particular, for any $\alpha\in\mathcal{M}$ and any morphism
$\tilde{f}:\alpha^*_*P\rightarrow Q$ in $\fSh$, the following diagram
is commutative:
\begin{equation*}
\label{etanatnot}
\xymatrix{
\Psi\Phi\alpha^*_*P\ar[rr]^{\eta_{\alpha^*_*P}} \ar[d]_{\Psi\Phi \tilde{f}} && \alpha^*_*P \ar[d]^{\tilde{f}}\\
\Psi\Phi Q \ar[rr]_{\eta_Q} && Q\;.
}
\end{equation*}
On the other hand, we need to establish the commutativity of the
diagrams
\begin{equation*}
\xymatrix{
\pPsi\pPhi P\ar[rr]^{\widetilde{\eta}_{P}} \ar[d]_{\pPsi\pPhi [\tilde{f},\alpha^*]} && P
\ar[d]^{[\tilde{f},\alpha^*]}\\
\pPsi\pPhi Q \ar[rr]_{\widetilde{\eta}_Q} && Q\;.
}
\end{equation*}
Using the commutativity of the first of the preceding two diagrams
 %Equation~\eqref{etanatnot} 
and the displayed formula in Step~(9), we obtain the desired:
\begin{align*}
\widetilde{\eta}_Q\circ (\pPsi\pPhi [\tilde{f},\alpha^*])&=
[\eta_Q,\id_\mN^*]\circ[\Psi\Phi \tilde{f},\alpha^*]\nonumber\\
&=[\eta_Q\circ(\Psi\Phi \tilde{f}),\alpha^*]\nonumber\\
&=[\tilde{f}\circ\eta_{\alpha^*_*P},\alpha^*]\nonumber\\
&=[\tilde{f}\circ(\alpha^*_*\eta_{P}),\alpha^*]\nonumber\\
&=[\tilde{f},\alpha^*]\circ[\eta_P,\id_\mN^*]\nonumber\\
&=[\tilde{f},\alpha^*]\circ\widetilde{\eta}_P.
\end{align*}
\item By \cite[Prop.~2.2]{HKMZ:PiecewisePrincipalComoduleAlgebras}),
  we have $\Phi\Psi=\id_{\ocovf}$. Hence, it is easy to see that the
  family of identity morphisms $(\id_{A},\id_\mN)$ in $\auxp$
  establishes a natural isomorphism between $\pPhi\pPsi$ and
  $\id_{\auxp}$.
\end{enumerate}
\end{proof}

Our next step is to define an equivalence relation on $\pfShp$. Let
$[\tilde{f},\alpha^*],[\tilde{g},\beta^*]:P\rightarrow Q$ be morphisms
in $\pfShp$. We say that they are equivalent
($[\tilde{f},\alpha^*]\sim[\tilde{g},\beta^*]$) if
$\tilde{f}_{\pinf}=\tilde{g}_{\pinf}$ as morphisms of algebras (c.f.\
the equivalence
relation on $\auxp$, Lemma~\ref{eqauxp}).   By
\cite[Proposition~II.8.1]{MacLane:Categories}, we know that the
quotient category $\pfShp/\!\!\sim$ exists.  Moreover, it is easy to see
that the relation $\sim$ preserves the compositions of morphisms.  Hence,
by the proof of \cite[Proposition~II.8.1]{MacLane:Categories}, we do
not need to extend the relation $\sim$ to form a quotient category.  Note
that  the equivalence class of the morphism $[\tilde{f},\alpha^*]$ in
$\pfShp$ can be represented by $\tilde{f}_{\pinf}$.  Therefore,
 the quotient
functor $\pfShp\rightarrow\pfShp\!/\!\!\sim$ is defined on morphisms
as
\begin{equation}\label{qfunct}
[\tilde{f},\alpha^*]\longmapsto \tilde{f}_{\pinf}.
\end{equation}
In other words,
\begin{equation}\label{words}
[\tilde{f},\alpha^*]_\sim:=\tilde{f}_{\pinf}.
\end{equation}

The final step to arrive at our classification of finite coverings by
finitely-supported flabby sheaves is as follows:
\begin{lemma}
  The functors $\pPsi:\auxp\rightarrow\pfShp$ and
  $\pPhi:\pfShp\rightarrow\auxp$ send equivalent morphisms to
  equivalent morphisms.  They descend to functors between
  quotient categories
\begin{equation}
\xymatrix{
\pfShp\ar[r]^-{\pPsi}\ar[d] & \auxp \ar[d]\\
\pfShp/\!\!\sim \ar[r]_-{\overline{\Psi}} & \auxp/\!\!\sim,
}\quad
\xymatrix{
\auxp\ar[r]^-{\pPhi}\ar[d] & \pfShp \ar[d]\\
\auxp/\!\!\sim \ar[r]_-{\overline{\Phi}} & \pfShp/\!\!\sim\,,
}
\end{equation}
establishing the equivalence of $\,\pfShp/\!\!\sim\,$ and
$\,\auxp/\!\!\sim$.
\end{lemma}
\begin{proof}
  Note that for any morphism $f$ in $\ocovf$ and any morphism
  $\tilde{f}$ in $\fSh$, we have the following equalities of algebra
  maps:
\begin{equation}
(\Psi f)_{\pinf}=f,\quad \Phi\tilde{f}=\tilde{f}_{\pinf}.
\end{equation}
It follows that, if $(f,\alpha)\sim(g,\beta)$ in $\auxp$, then
\begin{equation}
\pPsi(f,\alpha)=[\Psi f,\alpha^*]\sim[\Psi g,\beta^*]=\pPsi(g,\beta)
\end{equation}
in $\pfShp$. Similarly, if $[\tilde{f},\alpha^*]\sim[\tilde{g},\beta^*]$
in $\pfShp$, then
\begin{equation}
\pPhi[\tilde{f},\alpha^*]=(\Phi\tilde{f},\alpha)\sim(\Phi\tilde{g},\beta)
=\pPhi[\tilde{g},\beta^*].
\end{equation}
\end{proof}

Summarizing the foregoing results, we obtain the following
commutative diagram of functors:
\begin{equation}\label{functdiagcube}
\xymatrix{
& \cov \ar[rr] & & \pfShp/\!\!\sim\\
\covaux \ar[ru]^-{\G{Z}} \ar[rr]^>>>>>>>>>>>>>\sim & &
\auxp/\!\!\sim\ar[ru]^-{\overline{\Psi}} &\\
& \fSh \ar[uu] \ar[rr] & & \pfShp.  \ar[uu]\\
\ocovf \ar[uu] \ar[ru]^-\Psi \ar[rr] & & \auxp \ar[uu] \ar[ur]_-{\pPsi} &
}
\end{equation}
Using the above diagram, we immediately conclude the main result
of this article:
\begin{thm}\label{SheafandCoveringEquivalence}
The  assignments given for any 
 $(A,\Pi)\!\in\! Ob(\cov)$, $F\!\in\! Ob(\pfShp/\!\!\sim)$,
$f\in \text{\it Mor}(\cov)$, \ 
$[\tilde{f},\alpha^*]_\sim\in \text{\it Mor}(\pfShp/\!\!\sim)$, by the formulae
  \begin{align*}
    %\label{ggb2}
    (A,\Pi)& \quad\longmapsto\quad
  \left\{U \mapsto A/R^{{\ord{\Pi}}}(U)\right\}_{U\in\Top(\pinf)}
\in Ob(\pfShp/\!\!\sim),\\
    %\label{gga2}
    F& \quad\longmapsto\quad
  \left(F(\pinf),\ \{{\rm Res}_0(F),\ {\rm Res}_1(F),\
      \ldots,\ {\rm Res}_n(F),\ \ldots\}\right) \in Ob(\cov),\\
f& \quad\longmapsto\quad [\Psi(f),\alpha_f^*]_\sim %\Psi(f)_{\pinf}=f
\in\text{\it Mor}(\pfShp/\!\!\sim),
\\
[\tilde{f},\alpha^*]_\sim&\quad\longmapsto\quad
\tilde{f}_{\pinf}\in\text{\it Mor}(\cov),
  \end{align*}
 are equivalence
  functors  between the
  category $\cov$ of finite coverings of algebras
and the quotient category
  $\pfShp/\!\!\sim$ of the category of
finitely-supported flabby sheaves of algebras
  over $\pinf$ with extended morphisms.
  Here
  $(A,\ord{\Pi})$ is the image of $(A,\Pi)$ under an equivalence
functor
inverse  to $\G{Z}$, and $\alpha_f$ is a tame surjection defined as
in~\eqref{alphaf}.
\end{thm}
\noindent
Observe that the equivalence functors of the above theorem are, essentially, 
identity on morphisms. This is because, on both sides of the
equivalence, morphisms considered as input data are only algebra
homomorhisms (see \eqref{words} and Definition~\ref{Def:Coverings}).
 They do, however, satisfy quite different conditions
to be considered  morphisms in an appropriate category. Thus the essence
of the theorem is to re-interpret the natural defining conditions for
an algebra homomorphism to be a morphism of coverings to more
refined conditions that make it a morphism between sheaves. What we gain
this way is a functorial
 description of coverings by the more potent concept
of a sheaf. We know now that lattice operations applied to a covering
will again yield a covering.

We  end this section by stating
Theorem~\ref{SheafandCoveringEquivalence} in the classical setting of
the Gelfand-Neumark equivalence~\cite[Lem.~1]{GelfandNeumark:Hilbert}
between the category of compact Hausdorff spaces and the opposite
category of unital commutative C*-algebras.  Since the intersection of
closed ideals in a C*-algebra equals their product, the lattices of
closed ideals in C*-algebras are always distributive.  Therefore,
remembering that the epimorphisms of commutative unital C*-algebras
can be equivalently presented as the pullbacks of embeddings of
compact Hausdorff spaces, we obtain:
\begin{cor}\label{claco}
  The category of finite closed coverings of compact Hausdorff
spaces (see the beginning of this section)
is equivalent to the
  opposite of the quotient category $\pfShp/\!\!\sim$ of
finitely-supported flabby sheaves of
 commutative unital C*-algebras over $\pinf$ with extended morphisms.
\end{cor}

{\bf Acknowledgements:} This work was partially supported by the
Polish Government grants N201 1770 33 (PMH, BZ), 189/6.PRUE/2007/7
(PMH), and the Argentinian grant PICT 2006-00836 (AK). Part of
this article was finished during a visit of AK at the Max Planck
Institute in Bonn.  The Institute support and hospitality are gratefully
acknowledged.  We are very happy to thank the following people for
discussions and advise:
Paul F.\ Baum, Pierre Cartier, George Janelidze,  Tomasz
Maszczyk, and Jan Rudnik.  Finally, we would like to extend our deepest
gratitude to Chiara Pagani for all her work at the initial 
and final stages of
this paper.

\renewcommand{\baselinestretch}{1.15}

\end{document}